\documentclass{amsproc}

\usepackage{amsmath}
\usepackage{amsthm}
\usepackage{amssymb}
\usepackage{amscd}
\usepackage{graphicx}
\usepackage[matrix,arrow]{xy}

\copyrightinfo{2008}{American Mathematical Society}

\newtheorem{thm}{Theorem}[section]
\newtheorem{defi}[thm]{Definition}
\newtheorem{prop}[thm]{Proposition}%[section]
\newtheorem{conj}[thm]{Conjecture}%[section
\newtheorem{lemma}[thm]{Lemma}%[section]
\newtheorem{cor}[thm]{Corollary}%[section]
\newcommand{\Alg}{\mbox{{\rm Alg}}}

\newcommand{\SZ}{{\mathcal{Z}}^{\Delta}}
\newcommand{\Z}{\Bbb Z}
\newcommand{\C}{\Bbb C}
\newcommand{\R}{\Bbb R}

\newcommand{\K}{\Bbb K}

\newcommand{\RP}{\Bbb R\mbox{{\rm P}}}

\newcommand{\Map}{\mbox{{\rm Map}}}
\newcommand{\Rat}{\mbox{{\rm Rat}}}
\newcommand{\CP}{\Bbb C {\rm P}}

\newcommand{\dis}{\displaystyle}
\newcommand{\p}{\prime}

\newcommand{\E}{\tilde{E}}
\newcommand{\I}{\mbox{{\rm (i)}}}
\newcommand{\II}{\mbox{{\rm (ii)}}}
\newcommand{\III}{\mbox{{\rm (iii)}}}

\newcommand{\mo}{\mbox{{\rm (mod $2$)}}}
\numberwithin{equation}{section}

\begin{document}

\title[Spaces of algebraic maps]
{Spaces of algebraic maps from real projective spaces into complex projective spaces
}


%    author one information
\author{Andrzej Kozlowski}
\address{Tokyo Denki Univesity, Inzai, Chiba 270-1382 Japan}
\curraddr{}
\email{akoz@mimuw.edu.pl}
\thanks{}

%    author two information
\author{Kohhei Yamaguchi}
\address{The University of Electro-Communications,  Chofu Tokyo 182-8585 Japan}
\curraddr{}
\email{kohhei@im.uec.ac.jp}
\thanks{}

\subjclass[2000]{Primary 55P35, 55P10; Secondly 55P91.
\\
\hspace{6.5mm}{\it Key words and phrases.}  Spaces of algebraic maps, equivariant homotopy equivalence}

\date{}

%%%(Abstract)%%%%%%%
\begin{abstract}
We study the homotopy types of  spaces of algebraic (rational) maps  
from real projective spaces 
into complex projective spaces.
We showed in \cite{AKY1} that in this setting
the inclusion of the space of rational maps into the space of all continuous maps is
a homotopy equivalence. 
%%%%
 In this paper
we prove that  the homotopy types of the terms of the natural  \lq degree\rq\   filtration approximate closer 
and closer  the homotopy type
of the space of continuous maps  and obtain bounds that describe the closeness of the approximation in terms of the degree. 
%Complex conjugation induces the $\Z/2$-actions on the space of algebraic maps and that of continuous maps
%and we also prove the equivariant result for the above inclusion map.
%%
Moreover, we compute low dimensional homotopy groups  of these spaces.
These results combined with those of \cite{AKY1} can be formulated as a single statement about $\Bbb Z/2$-equivariant homotopy equivalence between 
these spaces,
%the spaces of 
%algebraic and continuous maps from real to complex projective spaces, 
where the $\Bbb Z/2$-action is induced by the complex conjugation.  This generalizes a theorem of \cite{GKY2}.  %%%%%%
\end{abstract}

\maketitle


%%%(Section 1)%%%%%
\section{Introduction.}
%%%%%%%%%%%%%%%%%%%%
%%%

\subsection{Summary of the contents}
Let $M$ and $N$ be manifolds with some additional structure, e.g holomorphic, symplectic, real algebraic etc. The relation between the topology of the space of continuous maps preserving this structure and that of the space of all continuous maps has long been an object of study in several areas of topology and geometry. Early examples  were provided by Gromov's h-principle for holomorphic maps \cite{Grom}. In these cases the manifolds are complex, the structure preserving maps are the holomorphic ones, and the spaces of holomorphic and continuous maps turn out to be homotopy equivalent. However,  in many other cases, the space of structure preserving maps approximates, in some sense, the space of all continuous ones and becomes homotopy equivalent to it only after some kind of stabilization.  A paradigmatic example of this type was given in a seminal paper of Segal \cite{Se}, where the space of rational (or holomorphic) maps of a fixed degree from the Riemann sphere to a complex projective space was shown to approximate the space of all continuous maps in homotopy, with the approximation becoming better as the degree increases. Segal's result was extended to a variety of other target spaces by various authors (e.g. \cite{BHM}). Although it has been sometimes stated that these phenomena are inherently related to complex or at least symplectic structures, real analogs of Segal's result were given in \cite{Se}, \cite{Mo1}, \cite{GKY2}, \cite{Va}. In fact, Segal  formulated the complex and real approximation theorems which he had proved  as a single statement involving equivariant equivalence, with respect to complex conjugation (see the remark after Proposition 1.4 of \cite{Se}). A similar idea was used in \cite{GKY2}, Theorem 3.7.  This theorem amounts to two equivariant ones, one of which is equivalent to (a stable version of)  Segal's equivariant one, while the other one is related to the \lq real version\rq\  of Segal's theorem proved in \cite{Mo1}. 

All the results mentioned above (except the ones involving Gromov's h-principle), assume that the domain of the mappings is one dimensional (complex or real). It is natural to try to generalize them to the situation where the domain is higher dimensional. That such generalizations might be possible was first suggested by Segal (see the remark under Proposition 1.3 of \cite{Se}). 
A large step in this direction appeared to have been made when Mostovoy
 \cite{Mo2}  showed that the homotopy types of spaces of holomorphic maps 
from $\CP^m$ to $\CP^n$
 (for $m\le n$) approximate the homotopy types of the spaces of continuous maps, with the approximation becoming better as the degree increases. Unfortunately Mostovoy's published argument contains several gaps.  A new version of the paper, currently only available from the author, appears to correct all the mistakes, with the main results remaining essentially unchanged. There are two major changes in the proofs. One is that the space $\Rat_f(p,q)$ of $(p,q)$ maps from $\CP^m$ to $\CP^n$  that restrict to a fixed map $f$ on a fixed hyperplane, used in section 2 of the published article is replaced by the space $\overline {\Rat}_f(p,q)$ of pairs of $n+1$-tuples of polynomials in $m$ variables that produce these maps. In the published version of the article it is assumed that these two spaces are homotopy equivalent, which is clearly not the case. However, they are homotopy equivalent after stabilisation, both being equivalent to $\Omega^{2m}\CP^n$ - the space of continuous maps that restrict to $f$ on a fixed hyperplane.
 The second important change is the introduction of a new filtration on the simplicial resolution $X^\Delta \subset \Bbb R^N\times Y $ of a map $h: X \to Y$ and an embedding $i: X \to \Bbb R^N = \Bbb C^{N/2}$. This filtration is defined by means of complex skeleta  (where the complex $k$-skeleton of a simplex in a complex affine space is the union of all its faces that are contained in complex affine subspaces of dimension at most $k$) and replaces the analogous \lq\lq real\rq\rq\ filtration in the arguments of section 4.

 In \cite{AKY1}  a variant of Mostovoy's idea  was applied to the case of algebraic maps from $\RP^m$ to $\RP^n$.  This leads naturally to the question whether one can generalize the $\Bbb Z/2$-equivariant Theorem 3.7 of \cite{GKY2} to an analogous equivariant equivalence of the spaces of algebraic maps and
continuous maps between projective spaces, in which the domain is either a real or a complex projective space of dimension $m>1$, the range is a  complex projective space of dimension greater or equal to that of the domain.
(Here, the $\Bbb Z/2$-action is induced by complex conjugation.)  

Note that Theorem 3.7 of \cite{GKY2} has two parts, in the first the domain being complex and in the second real. Clearly, to prove the  first part we need a  Mostovoy's complex theorem. We plan to consider this problem in a future paper. 
Here we concentrate on the second part, concerning equivariant algebraic maps from real projective spaces to complex projective ones. Our main theorem is new, but it uses the main result of   \cite{AKY1} and where  our arguments are very similar to those used in that paper we omit their details and refer the reader to  \cite{AKY1}. 
%%%%
\par\vspace{2mm}\par
%%%
In the remainder of this section we introduce our notation  and state the main definitions and theorems. 
\subsection{Notation and Main Results}

We first introduce notation which is analogous to the one used in \cite{AKY1}, the presence of $\Bbb C$ indicating that the complex case is being considered (i.e. maps take values in $\CP^n$ or polynomials have coefficients in $\Bbb C$).
\par
Let $m$ and $n$ be positive integers such that
$1\leq m <2\cdot (n+1)-1$.
We choose  ${\bf e}_m=[1:0:\cdots :0]\in \RP^m$ and
${\bf e}_n^{\p}=[1:0:\cdots :0]\in\Bbb \CP^n$
as the base points of $\RP^m$ and  $\Bbb  \CP^n$, respectively.
%For $d(\K)\leq m <d(\K) \cdot (n+1)-1$,
Let $\Map^*(\RP^m,\Bbb  \CP^n)$ denote the space
consisting of all based maps
$f:(\RP^m,{\bf e}_{m})\to (\Bbb \CP^n,{\bf e}^{\p}_n)$. 
%%%
When $m\not= 1)$,
%$d(\K)\leq m< d(\K)\cdot (n+1)-1$, 
we denote
by $\Map_{\epsilon}^* (\RP^m,\CP^n)$
the corresponding path component
of $\Map^* (\RP^m,\CP^n)$
for each
$\epsilon \in \Z/2=\{0,1\}=
\pi_0(\Map^*(\RP^m,\Bbb  \CP^n))$
(\cite{CS}). 
%the corresponding path component
%of $\Map^* (\RP^m,\KP^n)$
%(\cite{CS}).
Similarly, let $\Map (\RP^m,\CP^n)$ denote the space
of all free maps $f:\RP^m\to\CP^n$ and
$\Map_{\epsilon}(\RP^m,\CP^n)$ the corresponding path component
of $\Map (\RP^m,\CP^n)$.
%%%%%
\par
We shall use the symbols $z_i$ when we refer to complex valued coordinates or  variables or when we refer to complex and real valued ones at the same time while the notation $x_i$ will be restricted to the purely real case.
%%%%
\par\vspace{2mm}\par
%%%
A  map $f:\RP^m\to \CP^n$ is called a
{\it algebraic map of the degree }$d$ if it can be represented as
a rational map of the form
$f=[f_0:\cdots :f_n]$ such that
 $f_0,\cdots ,f_n\in\Bbb C [z_0,\cdots ,z_m]$ are homogeneous polynomials of the same degree $d$
with no common {\it real } roots except ${\bf 0}_{m+1}=(0,\cdots ,0)\in\R^{m+1}$.
 %%%
 We  denote by $\Alg_d(\RP^m,\CP^n)$ (resp. $\Alg_d^*(\RP^m,\CP^n)$) the space
 consisting of all (resp. based) algebraic maps $f:\RP^m\to \CP^n$
 of degree $d$.
 It is easy to see that there are inclusions $\Alg_d(\RP^m,\CP^n)\subset \Map_{[d]_2}(\RP^m,\CP^n)$
 and
 $\Alg^*_d(\RP^m,\CP^n)\subset \Map^*_{[d]_2}(\RP^m,\CP^n)$,
 where $[d]_2\in\Z/2=\{0,1\}$ denotes the integer $d$ mod $2$.
%%%
\par
Let
$ A_{d}(m,n)(\Bbb C)$ denote  the space consisting of all $(n+1)$-tuples 
$(f_0,\cdots ,f_n)\in \Bbb C[z_0,\cdots ,z_m]^{n+1}$
of  homogeneous polynomials of degree $d$  with coefficients in $\Bbb C$  and without non-trivial common real roots
(but possibly with non-trivial common {\it non-real} ones).
\par  
Let $ A_{d}^{\Bbb C}(m,n)\subset A_d(m,n)(\Bbb C)$ be the subspace consisting of $(n+1)$-tuples 
$(f_0,\cdots ,f_n)\in A_d(m,n)(\Bbb C)$
such that the coefficient of $z_0^d$ in $f_0$ is 1 and $0$ in the other $f_k$'s
$(k\not= 0$).
%%%
Then there is a natural surjective projection map
%%(1.1)%%
\begin{equation}\label{Psi}
\Psi_d^{\Bbb C}:A_d^{\Bbb C}(m,n) \to \Alg_d^*(\RP^m,\CP^n).
\end{equation}
%%%%%%%%%% 
If $d=2d^*\equiv 0$ $\mo$ is an even positive integer,
we also have a natural projection map
%%(1.2)%%
\begin{equation}\label{}
j_d^{\Bbb C}:A_d^{\Bbb C}(m,n) \to \Map^*(\RP^m,\Bbb C^{n+1}\setminus\{{\bf 0}\})
\simeq
\Map^*(\RP^m,S^{2n+1})
\end{equation}
%%%%
defined by
$$
j_d^{\Bbb C}(f)([x_0:\cdots :x_m])=
\Big(\frac{f_0(x_0,\cdots ,x_m)}{(\sum_{k=0}^mx_k^2)^{d^*}},\cdots
,\frac{f_n(x_0,\cdots ,x_m)}{(\sum_{k=0}^mx_k^2)^{d^*}}\Big)
$$
for $f=(f_0,\cdots ,f_n)\in A_d^{\Bbb C}(m,n)$.
Note that the map
$j_d^{\Bbb C}$ is well defined only if $d\geq 2$ is an even integer.
\par\vspace{2mm}\par
For $m\geq 2$ and $g\in \Alg_d^*(\RP^{m-1},\CP^n)$ a fixed algebraic map,
we denote by
$\Alg_d^{\Bbb C}(m,n;g)$ and $F_d^{\Bbb C}(m,n;g)$  the spaces defined by
$$
\begin{cases}
\Alg_d^{\Bbb C}(m,n;g) &=\ \{f\in \Alg_d^*(\RP^m,\CP^n):f\vert \RP^{m-1}=g\},
\\
F_d^{\Bbb C}(m,n;g) & =\ \{f\in \Map_d^*(\RP^m,\CP^n):f\vert \RP^{m-1}=g\}.
\end{cases}
$$
%%%
\par
It is well-known that there is a homotopy equivalence
$F_d^{\Bbb C}(m,n;g)\simeq  \Omega^m\CP^n$ (\cite{Sasao}).
Let $A_d^{\Bbb C}(m,n;g)\subset A_d^{\Bbb C}(m,n)$ denote the subspace
given 
by 
$$A_d^{\Bbb C}(m,n;g)=(\Psi_d^{\Bbb C})^{-1}(\Alg^{\Bbb C}_d(m,n;g)).
$$
%%%
%%%%
Observe that if an algebraic map $f\in \Alg_d^*(\RP^m,\CP^n)$ can be represented as
$f=[f_0:\cdots :f_n]$ for some 
$(f_0,\cdots ,f_n)\in A_d^{\Bbb C}(m,n)$ then  the same map can also be represented as
$f=[\tilde{g}_mf_0:\cdots :\tilde{g}_{m}f_n]$, where
$\tilde{g}_{m}=\sum_{k=0}^mz_k^2$.
So there is an inclusion 
$\Alg_d^*(\RP^m,\CP^n)\subset \Alg_{d+2}^*(\RP^m,\CP^n)$
and we can define
{\it the stabilization map}
$s_d:A_d^{\Bbb C}(m,n)\stackrel{}{\to} A_{d+2}^{\Bbb C}(m,n)$
by $s_d(f_0,\cdots ,f_n)=(\tilde{g}_mf_0,\cdots ,\tilde{g}_mf_n).$
It is easy to see that there is a commutative diagram
$$
\begin{CD}
A_d^{\Bbb C}(m,n) @>s_d>> A_{d+2}^{\Bbb C}(m,n)
\\
@V{\Psi_d^{\Bbb C}}VV @V{\Psi_{d+2}^{\Bbb C}}VV
\\
\Alg^{*}_d(\RP^m,\CP^n) @>\subset>> \Alg_{d+2}^{*}(\RP^m,\CP^n)
\end{CD}
$$
A map $f\in \Alg_d^*(\RP^m,\CP^n)$ is  called an algebraic map of
{\it minimal degree} $d$ if 
$f\in \Alg_d^*(\RP^m,\CP^n)\setminus \Alg_{d-2}^*(\RP^m,\CP^n)$.
It is easy to see that
if $g\in\Alg_d^*(\RP^{m-1},\CP^n)$ is an algebraic map of minimal degree $d$, then
the restriction
%%(1.3)%%
\begin{equation}\label{restriction}
\Psi_d^{\Bbb C}\vert A_d^{\Bbb C}(m,n;g):A_d^{\Bbb C}(m,n;g)\stackrel{\cong}{\rightarrow}
\Alg_d^{\Bbb C}(m,n;g)
\end{equation}
is a homeomorphism.
%%%
Let 
%%(1.4)%%
\begin{equation}
\begin{cases}
i_{d,\Bbb C}:\Alg_d^*(\RP^m,\CP^n)\stackrel{\subset}{\rightarrow} \Map_{[d]_2}^*(\RP^m,\CP^n)
\\
i_{d,\Bbb C}^{\p}:\Alg_d^{\Bbb C}(m,n;g)\stackrel{\subset}{\rightarrow} F(m,n;g)\simeq \Omega^m\CP^n
%\\
%i_d^{\Bbb K}=i_{d,\K}\circ \Psi_d:A_d^{\Bbb K}(m,n)\to \Map_{[d]_2}^*(\RP^m,\KP^n)
\end{cases}
\end{equation}
%%%%
denote the inclusions and let
%%(1.5)%%
\begin{equation}
i_d^{\Bbb C}=i_{d,\Bbb C}\circ \Psi_d^{\Bbb C}:A_d^{\Bbb C}(m,n)\to \Map_{[d]_2}^*(\RP^m,\CP^n).
\end{equation}
be the natural projection.
For a connected space $X$, 
let $F(X,r)$ denote the configuration space  of distinct $r$ points in $X$.
The symmetric group $S_r$ of $r$ letters acts on $F(X,r)$ freely by permuting 
coordinates. Let $C_r(X)$ be the configuration space of unordered 
$r$-distinct  points in $X$ given by 
$C_r(X)=F(X,r)/S_r$.
Note that there is a stable homotopy equivalence
%(the Snaith splitting)
$\Omega^mS^{m+l}\simeq_s \bigvee_{r=1}^{\infty}F(\R^m,r)_+\wedge_{S_r}\big(\bigwedge^rS^l\big)$
(\cite{Sn}), and it is known that
there is an isomorphism
$H_k(F(\R^m,r)_+\wedge_{S_r}\big(\bigwedge^rS^l\big),\Z)
\cong H_{k-rl}(C_r(\R^m),(\pm \Z)^{\otimes r})$ for $k,l\geq 1$
(\cite{CLM},  \cite{Va}),
where
$\bigwedge^rX=X\wedge \cdots \wedge X$
($r$ times).
\par
Let $G^d_{m,N}$ denote the abelian group $G^d_{m,N}=\bigoplus_{r=1}^{\lfloor \frac{d+1}{2}\rfloor}
H_{k-(N-m)r}(C_r(\R^m),(\pm \Z)^{\otimes (N-m)} )$, where
the meaning of  $(\pm \Z)^{\otimes (N-m)}$  is the same as in \cite{Va}.

Let   $D_{\Bbb K}(d;m,n)$
be the positive integer defined by
%%(1.6)%%

$$D_{\K}(d;m,n)=
\begin{cases}
(n-m)\big(\lfloor \frac{d+1}{2}\rfloor  +1\big) -1
& \mbox{if } \K =\R,
\\
(2n-m+1)\big(\lfloor \frac{d+1}{2}\rfloor  +1\big) -1
& \mbox{if } \K =\C,
\end{cases}
$$
%%%%%%%
, and 
$\lfloor x\rfloor$ is the integer part of a real number $x$.
Note that $D_{\C}(d;m,n)=D_{\R}(d;m,2n+1)$.
%%%%%%
\par\vspace{2mm}\par
%%%
%%
First, recall 
the following 3 results. Here we use the notation of \cite{AKY1}.

%Since the stable result follows from the unstable one we only state the latter.
%
%%%(Theorem 1.1; KY1 etc)%%%%%%%
\begin{thm}[\cite{KY1}, \cite{Y5}]\label{thm: A1}
%%%%%%]
If $n\geq 2$ and then
the natural projection
$i_{d}^{\R}:A_d^{\R}(1,n)\to  \Map^*_{[d]_2}(\RP^1,\RP^n)\simeq \Omega S^{n}$
is a homotopy equivalence
up to dimension $D_1(d,n)=(d+1)(n-1)-1$.
\end{thm}
%%(End of Theorem 1.1)%%%%
%%
%%(Theorem 1.2: AKY1)%%
\begin{thm}[\cite{AKY1}]
\label{thm: AKY1-I}
Let $2\leq m <n$ be integers and 
let $g\in \Alg_d^*(\RP^{m-1},\RP^n)$ be an algebraic map of minimal degree $d$.
%%%
\begin{enumerate}
%%(i)%%
\item[$\I$]
The inclusion
$i_{d,\R}^{\p}:\Alg_d^{\R}(m,n;g)\to F_d(m,n;g)\simeq \Omega^mS^n$
is a homotopy equivalence through dimension $D_{\R}(d;m,n)$ if $m+2\leq n$
and a homology equivalence through dimension $D_{\R}(d;m,n)$ if $m+1=n$. 
%%(ii)%%%
\item[$\II$]
For any $k\geq 1$, 
$H_k(\Alg_d^{\R}(m,n;g),\Z)$ contains the subgroup
$G^d_{m,n}$ as a direct summand.
%Here $G^d_{m,N}$ denote the abelian group defined by
%$$
%G^d_{m,N}=\bigoplus_{r=1}^{\lfloor \frac{d+1}{2}\rfloor}
%H_{k-(N-m)r}(C_r(\R^m),(\pm \Z)^{\otimes (N-m)} ),
%$$
%where the meaning of  $(\pm \Z)^{\otimes (N-m)}$  is the same as in \cite{Va}.
%%%
%%%%%%
\end{enumerate}
\end{thm}
%%%
%%(Theorem 1.3)%%%%
\begin{thm}[\cite{AKY1}]
\label{thm: AKY1-II}
%%%%
If $2\leq m <n$ and $d\equiv 0$ $\mo$ are positive integers,
$$\begin{cases}
j_d^{\R}:A_d^{\R}(m,n)\to \Map^*(\RP^m,S^n)
\\
i_d^{\R}:A_d^{\R}(m,n)\to \Map_0^*(\RP^m,\RP^n)
\end{cases}
$$
are homotopy equivalence through dimension $D_{\R}(d;m,n)$ if
$m+2\leq n$ and homology equivalences through dimension
$D_{\R}(d;m,n)$ if $m+1=n$.
\end{thm}
%%%%
%%(Remark)%%%
\par
{\it Remark. }
%%%
A map $f:X\to Y$ is called {\it a homotopy} (resp. {\it a homology}) {\it equivalence up to dimension} $D$ if
$f_*:\pi_k(X)\to \pi_k(Y)$ (resp.$ f_*:H_k(X,\Z)\to H_k(Y,\Z)$) is an isomorphism for any
$k<D$ and an epimorphism for $k=D$.
Similarly, it is called {\it a homotopy} (resp. {\it a homology}) {\it equivalence through dimension} $D$ if
$f_*:\pi_k(X)\to \pi_k(Y)$ (resp.$ f_*:H_k(X,\Z)\to H_k(Y,\Z)$) is an isomorphism for any $k\leq D$.
%%
%%
%%%%%%%%%
\par\vspace{3mm}\par
%%%
In this paper, from now on let $m,n\geq 2$ be positive integers. Our main results are as follows.
%%
%%(Theorem 1.4: Theorem I)%%
\begin{thm}\label{thm: I}
%%%%%%%
Let  $2\leq m\leq 2n$, and 
let $g\in \Alg_d^*(\RP^{m-1},\CP^n)$ be an algebraic map of minimal degree $d$.
%%%
\begin{enumerate}
%%(i)%%
\item[$\I$]
%%%
The inclusion
$i_{d,\C}^{\p}:\Alg_d^{\C}(m,n;g)\to F_d(m,n;g)\simeq \Omega^mS^{2n+1}$
is a homotopy equivalence through dimension $D_{\C}(d;m,n)$ if
$m<2n$ and a homology equivalence through dimension $D_{\C}(d;m,n)$
if $m=2n$. 
%%(ii)%%%
\item[$\II$]
For any $k\geq 1$, 
$H_k(\Alg_d^{\C}(m,n;g),\Z)$ contains the subgroup
$G^d_{m,2n+1}$ as a direct summand.
%%%%%%
\end{enumerate}
\end{thm}
%%%(End of Theorem 1.4)%%

%%(Theorem 1.5)%%%%
\begin{thm}\label{thm: II}
%%%%
If $2\leq m\leq 2n$ and $d\equiv 0$ $\mo$ are positive integers,
$$
\begin{cases}
j_d^{\C}:A_d^{\C}(m,n)\to \Map^*(\RP^m,S^{2n+1})
\\
i_d^{\C}:A_d^{\C}(m,n)\to \Map^*_0(\RP^m,\CP^n)
\end{cases}
$$
are homotopy equivalences through dimension $D_{\C}(d;m,n)$
if 
$m<2n$ and homology equivalences through dimension $D_{\C}(d;m,n)$
if $m=2n$.
\end{thm}
%%%%
%%%(Theorem 1.6)%%
%\begin{thm}\label{thm: III}
%%%%%%%%%
%If $2\leq m\leq 2n$ and $d=2d^*\equiv 0$ $\mo$ is a positive integer,  
%the map
%$i_d^{\C}:A_d^{\C}(m,n)\to \Map^*_0(\RP^m,\CP^n)$
%is a homotopy equivalence through dimension $D_{\C}(d;m,n)$ if $m<2n$ and
%a homology equivalence through dimension $D_{\C}(d;m,n)$
%if $m=2n$.
%%%%
%\end{thm}
%%%%%%%

%%%(Corollary 1.6)%%
\begin{cor}\label{cor: III}
%%%%%
If $2\leq m\leq 2n$ and $d\equiv 0$ $\mo$ are positive integers,  
the stabilization map
$s_d:A_d^{\C}(m,n)\to A_{d+2}^{\C}(m,n)$ 
is a homotopy equivalence through dimension
$D_{\C}(d;m,n)$ if
$m<2n$ and a homology equivalence through dimension $D_{\C}(d;m,n)$
if $m=2n$.
\end{cor}
%%%%%

%%
Note that the
complex conjugation on $\C$ naturally induces $\Z/2$-actions on
the spaces $\Alg_d^{\C}(m,n;g)$ and $A_d^{\C}(m,n)$.
In the same way it also induces a $\Z/2$-action on $\CP^n$ and this action extends to actions on the spaces
$\Map^*(\RP^m,S^{2n+1})$ and $\Map^*_{\epsilon}(\RP^m,\CP^n)$,
 where we identify 
$S^{2n+1}=\{(w_0,\cdots ,w_n)\in\C^{n+1}:\sum_{k=0}^n\vert w_k\vert^2 =1\}$
and regard $\RP^m$ as a $\Z/2$-space with the trivial $\Z/2$-action.
Since the maps $i^{\p}_{d,\C}$, $j_d^{\C}$, $i_d^{\C}$ are
$\Z/2$-equivariant and
$(i^{\p}_{d,\C})^{\Z/2}=i^{\p}_{d,\R}$, $(j_d^{\C})^{\Z/2}=j_d^{\R}$, 
$(i_d^{\C})^{\Z/2}=i_d^{\R}$, we easily obtain the following result.

%%(Corollary 1.7)%%
\begin{cor}\label{cor: IV}
%%%
Let $2\leq m\leq 2n$ and $d\geq 1$ be  positive integers.
\begin{enumerate}
%%(i)%%
\item[$\I$]
The inclusion map
$i_{d,\C}^{\p}:\Alg_d^{\C}(m,n;g)\to F_d(m,n;g)\simeq \Omega^mS^{2n+1}$
is a $\Z/2$-equivariant homotopy equivalence through dimension $D_{\R}(d;m,n)$ if
$m<2n$ and a $\Z/2$-equivariant homology equivalence through dimension $D_{\R}(d;m,n)$
if $m=2n$.
%%(ii)%%
\item[$\II$]
If $d\equiv 0$ $\mo$, then
$$
\begin{cases}
j_d^{\C}:A_d^{\C}(m,n)\to \Map^*(\RP^m,S^{2n+1})
\\
i_d^{\C}:A_d^{\C}(m,n)\to \Map^*_0(\RP^m,\CP^n)
\end{cases}
$$
are $\Z/2$-equivariant homotopy equivalences through dimension $D_{\R}(d;m,n)$
if 
$m<2n$ and $\Z/2$-equivariant homology equivalences through dimension $D_{\R}(d;m,n)$
if $m=2n$.
%%%
\end{enumerate}
\end{cor}
%%%%%%%%%%%
\par
{\it Remark. }
%%%
Let $G$ be a finite group and let
$f:X\to Y$ be a $G$-equivariant map.
Then a map $f:X\to Y$ is called 
{\it a $G$-equivariant homotopy }(resp. {\it homology})
{\it equivalence through dimension }$D$ if the induced
homomorphism
$f_*^{H}:\pi_k(X^{H})\stackrel{\cong}{\rightarrow}\pi_k(Y^{H})$
(resp.
$f_*^{H}:H_k(X^{H},\Z)\stackrel{\cong}{\rightarrow}H_k(Y^{H},\Z)$)
are isomorphisms for any $k\leq D$ and any subgroup $H\subset G$.
%%%
\par\vspace{2mm}\par
%%%%%%

Of course we would also like to understand the cases $d\equiv 1$ $\mo$.
The homotopy type of
$\Alg_d^{*}(\RP^m,\CP^n)$ appears hard to investigate in general. 
However, for $d=1$, 
$\Psi_1^{\K}:A^{\C}_1(m,n)\stackrel{\cong}{\rightarrow}\Alg_1^{*}(\RP^m,\CP^n)$
is a homeomorphism and we can prove the following results.
%%%(Theorem 1.8)%%
\begin{thm}\label{thm: V}
%%%%%%
\begin{enumerate}
\item[$\I$]
If $2\leq m< 2n$,
the inclusion
$$
i_{1,\C}:\Alg_1^{*}(\RP^m,\CP^n)
\to
\Map_1^*(\RP^m,\CP^n)
$$
is a homotopy equivalence up to dimension
$D_{\C}(1;m,n)=4n-2m+1$.
%%(ii)%%
\item[$\II$]
If $m=2n\geq 4$, the inclusion $i_{1,\C}$ induces an isomorphism
$$
(i_{1,\C})_*:\pi_1(\Alg_1^{*}(\RP^{2n},\CP^n))
\stackrel{\cong}{\longrightarrow}
\pi_1(\Map^*_1(\RP^{2n},\CP^n))\cong\Z/2 .
$$
\end{enumerate}
\end{thm}
%%%%%
%
%%(Corollary 1.9)
\begin{cor}\label{cor: V}
%%%%%%%
\begin{enumerate}
%%(i)%%
\item[$\I$]
If $2\leq m< 2n$ and $d\equiv 0$ $\mo$ are positive integers,  
the space $\Alg_d^{*}(\RP^m,\CP^n)$
is $(2n-m)$-connected and
$$
\pi_{2n-m+1}(\Alg_{d}^{*}(\RP^m,\CP^n))
\cong
\begin{cases}
\Z & \mbox{if }\ m\equiv 1\ \mo ,
\\
\Z/2 & \mbox{if }\ m\equiv 0\ \mo .
\end{cases}
$$
%%(ii)%%%
\item[$\II$]
If $2\leq m\leq 2n$ and $\epsilon\in\{0,1\}$,  the
two spaces $\Alg_1^{*}(\RP^m,\CP^n)$ and
$\Map_{\epsilon}^*(\RP^m,\CP^n)$ are $(2n-m)$-connected, and
%%%%
\begin{eqnarray*}
\pi_{2n-m+1}(\Alg_{1}^{*}(\RP^m,\CP^n))
&\cong&
\pi_{2n-m+1}(\Map_{\epsilon}^*(\RP^m,\CP^n))
\\
&\cong&
\begin{cases}
\Z & \mbox{if }\ m\equiv 1\ \mo ,
\\
\Z/2 & \mbox{if }\ m\equiv 0\ \mo .
\end{cases}
\end{eqnarray*}
\end{enumerate}
%%%
\end{cor}
%%%%%%%%%%
\par
{\it Remark. }
We conjecture that $\pi_1(\Alg^{*}_{d}(\RP^{m},\CP^n))=\Z/2$
if $m=2n\geq 4$ and $d\equiv 0$ $\mo$, but at this time
we cannot prove this.
%%%%
\par\vspace{2mm}\par
This paper is organized as follows.
In section 2, we consider the space of algebraic maps
$\RP^m\to \CP^n$ and recall the stable Theorem obtained in
\cite{AKY1}.
In section 3 we study simplicial resolutions and
the spectral sequences induced from them and prove Theorem \ref{thm: I}.
In section 4 we prove Theorem \ref{thm: II} and
Corollary \ref{cor: III}
by using Theorem \ref{thm: AKY1-II}, and in section 5 we study the homotopy type of the stabilized space
$A^{\C}_{\infty +\epsilon}(m,n)$ for $\epsilon =0$.
Finally in section 6, we investigate
 $A_d^{\C}(m,n)$ for the case $d=1$ and prove Theorem \ref{thm: V}
and Corollary \ref{cor: V}.

%%%%%%(SECTION 2)%%%%%%%%%%%%%%
\section{Spaces of algebraic maps.}
%%%%%%%%%%%%%%%%%%%%%%
%%%%
%%%%%%%%%

An algebraic  map $f:\RP^m\to\CP^n$ can always be 
 represented as
$f=[f_0:f_1:\cdots :f_n]$, where
$f_0,\cdots ,f_n\in \C [z_0,z_1,\cdots ,z_m]$
are homogeneous polynomials of the same degree $d$
with no common {\it real} root other than
${\bf 0}_{m+1}=(0,\cdots ,0)\in\R^{m+1}$
(but possibly with common {\it non-real} roots).
%%%%
%\par
%%%% 
%We will refer to a algebraic map represented in this way as 
%{\it an algebraic map of degree} $d$.  
%Of course a representation of a algebraic map as an algebraic map of degree $d$  
%is not unique; 
%by multiplying all the components by a nowhere-vanishing polynomial 
%we obtain a a higher degree algebraic representation of the same map. 
%For a algebraic map $f$, we shall refer to the smallest possible degree of the polynomials in such a representation as the {\it minimal degree} of $f$.  Note that unlike degree, which depends on a representation of an algebraic map, the minimal degree depends only on the (algebraic) map itself. 
\par
%%%%%
Clearly, an element of $\Alg_d^*(\RP^m,\CP^n)$  
can always be represented in the form $f=[f_0:f_1:\cdots :f_n]$, 
such that the coefficient of $z_0^d$ in $f_0$
is $1$
and in the other polynomials $f_i$ $(i\not= 0)$ $0$.
%%%%
In general, such a representation is also not unique. 
For example, if we multiply all polynomials $f_i$ by two different homogeneous polynomials in $z_0,z_1,\cdots ,z_m$, which contain a power of $z_0$ with coefficient   $1$  
and are always positive on $\RP^m$, we will obtain two distinct representations of the same algebraic map.
So the map
$\Psi_d^{\C}:A_d^{\C}(m,n)\to \Alg_{d}^*(\RP^m,\CP^n)$ is a surjective projection map.
It is easy to see that any fiber of $\Psi_d^{\C}$ is homeomorphic to the space
consisting of non-negative positive homogeneous polynomial functions of some
fixed even degree. So it is convex and contractible.
Hence, it seems plausible to expect the following may be true.
%%(Conjecture 2.1)%%
\begin{conj}\label{conj: Psi}
The map
$\Psi_d^{\C}:A_d^{\C}(m,n) \to \Alg_d^*(\RP^m,\CP^n)$
is a homotopy equivalence.
\end{conj}
%%%
Although  we cannot prove this conjecture, we
show in section 6 that it is true if  $d\to \infty$ through even integers .
%%%
\par\vspace{2mm}\par
%%% 
We always have $\Alg_d^*(\RP^m,\CP^n) \subset \Alg_{d+2}^*(\RP^m,\CP^n)$
and $\Alg_d(\RP^m,\CP^n) \subset \Alg_{d+2}(\RP^m,\CP^n)$, 
because
$[f_0:f_1:\cdots :f_n]= 
[\tilde{g}_{m} f_0:\tilde{g}_mf_1:\cdots :\tilde{g}_{m} f_n]$.
%%%%%
\par\vspace{2mm}\par
%%%%%%%(Definition)%%%%
%{\bf Definition. }
%%%%%
%%(Definition 2.2)%%%%
\begin{defi}
{\rm
For $\epsilon \in\{0,1\}$,  define  subspaces 
$\Alg_{\epsilon}^*(m,n)\subset\Map_{\epsilon}^*(\RP^m,\CP^n)$
and
$\Alg_{\epsilon}(m,n)\subset\Map_{\epsilon}(\RP^m,\CP^n)$
by
%%%%%%%%
$$
\begin{cases}
\Alg^*_{\epsilon}(m,n)   &=\ \bigcup_{k=1}^{\infty}
\Alg^*_{\epsilon+2k}(\RP^m,\CP^n),
\\
\Alg_{\epsilon}(m,n)   &=\ \bigcup_{k=1}^{\infty}
\Alg_{\epsilon+2k}(\RP^m,\CP^n).
\end{cases}
$$
}
\end{defi}
%%%(Theorem 2.3)%%%%%%
\begin{thm}\label{thm: stable} 
%%%%%%%%
If $1\leq m\leq 2n$ and $\epsilon=0$ or $1$, 
the inclusion maps
$$
\begin{cases}
i:\Alg_{\epsilon}^*(m,n)\stackrel{\simeq}{\rightarrow}
 \Map_{\epsilon}^*(\RP^m,\CP^n)
\\
j:\Alg_{\epsilon}(m,n)\stackrel{\simeq}{\rightarrow}
 \Map_{\epsilon}(\RP^m,\CP^n)
\end{cases}
$$
are homotopy equivalences.
\end{thm}
%%%%%%%%
\begin{proof}
It follows from  [\cite{AKY1}, Theorem 2.3] that $j$ is a homotopy equivalence.
The statement and the proof in it
 are valid also for spaces of  
based maps, and we can show
 that $i$ is also a homotopy equivalence.
\end{proof}
%%%%%

%%%(SECTION 3)%%%%
\section{Spectral sequences of the Vassiliev type.}\label{section 3}
%%%%%%%%%%%%%%%%%%
%%
From now on, we assume $2\leq m\leq 2n$ and
let $g\in \Alg_d^*(\RP^{m-1},\CP^n)$ be a fixed algebraic map of
minimal degree $d$, such that
$g=[g_0:\cdots :g_n]$ with
$(g_0,\cdots ,g_n)\in A_d^{\C}(m-1,n)$.
Note that $(g_0,\cdots ,g_n)$ is uniquely determined by $g$ (because of the minimal degree condition).
%%%
\par\vspace{2mm}\par

%%%
Let $\mathcal{H}_d
\subset \C [z_0,\cdots ,z_m]$
denote the subspace consisting of all
homogeneous polynomials of degree $d$.
For $\epsilon \in \{0,1\}$,
let $H_d^{\epsilon}\subset \mathcal{H}_d$ be the subspace
consisting of all homogeneous polynomials $f\in \mathcal{H}_d$
such that the coefficient of $(z_0)^d$ of $f$ is
$\epsilon$.
%%

%%%%%
Since $A_d^{\C}(m,n)$ is the space
consisting of all $(n+1)$-tuples
$(f_0,\cdots ,f_n)\in A_d(m,n)(\C)$
such that the coefficient of $z_0^d$ in $f_0$ is $1$ and
those of other $f_k$'s are all zero,
$A_d^{\C}(m,n)\subset \mathcal{H}_d^0\times (\mathcal{H}_d^1)^n$.
Note that
$\mathcal{H}_d^0\times (\mathcal{H}_d^1)^n$ is an affine space of real dimension of
$N_d=2(n+1)\Big(\binom{m+d}{m}-1\Big)$.

%%%
\par\vspace{2mm}\par
%%%
%%%%
%%%
Next, we set $B_k=\{g_k+z_mh:h\in \mathcal{H}_{d-1}\}$ $(k=0,1,\cdots ,n)$ and
define the subspace $A_d^*\subset \mathcal{H}_d^0\times (\mathcal{H}_d^1)^n$ by
$A_d^*=B_0\times B_1\times \cdots \times B_n$.
Note that
$A_d^*$ is an affine space of real dimension $N_d^*=2(n+1)\binom{m+d-1}{m}.$
%%%
\par\vspace{2mm}\par
%%(Definition 3.1)%%
\begin{defi}
%{\bf Definition. }
%%%%%%%%%%%%%%%
%Define the space
{\rm 
Let $A_d^{\C}(m,n;g)\subset A_d^*$ be the subspace 
$A_d^{\C}(m,n;g)=A_d^*\cap A_d^{\C}(m,n).$
Let
 $\Sigma_d\subset A_d$
%and $\Sigma_d \subset A_d$
denote the discriminant of $A_d^{\C}(m,n;g)$ in $A_d^*$ 
%or of
%$A_d^{\C}(m,n)$ in $A_d$
defined by
}
%$$
$\Sigma_d=A_d^*\setminus A_d^{\C}(m,n;g).$
%\quad
%\Sigma_d =A_d\setminus A_d^{\C}(m,n).
%$$
\end{defi}
%%%%%%%%%%%%%
%\par\vspace{2mm}\par

%%%
Since $g\in \Alg_d^*(\RP^{m-1},\RP^n)$ has minimal degree $d$, 
clearly  the restriction
%%%%(7)%%
\begin{equation}\label{Psi1}
\Psi_d^{\C}\vert_{A_d^{\C}(m,n;g)}:A_d^{\C}(m,n;g)\stackrel{\cong}{\rightarrow}
\Alg_d^{\C}(m,n;g)
\end{equation}
 is a homeomorphism.
%%%%%%
\par\vspace{2mm}\par
%%%%%%%%%%%%%%%%%%%%%%%

%%%%(Lemma 3.2)%%
\begin{lemma}\label{lemma: number}
$\I$
%%%
If $(f_0,\cdots ,f_n)\in \Sigma_d$ and
${\bf x}=(x_0,\cdots ,x_m)\in \R^{m+1}$ is a non-trivial common root
of $f_0,\cdots ,f_n$, then $x_m\not= 0$.
%%%
\par
%%%
$\II$
%If $(f_0,\cdots ,f_n)\in \Sigma_d$, then the number of the
%distinct common real roots of $\{f_0,\cdots ,f_n\}$ is at most $d^m$.
%%%%
%\par
%%%%(iii)%%%
%$\III$
 $A_d^{\C}(m,n;g)$ and
$A_d^{\C}(m,n)$ are simply connected if $m<2n$.
\end{lemma}
%%%%%%%%%%%%%
%%%
\begin{proof}
The proof is completely analogous to that of [\cite{AKY1}, Lemma 4.1].
%(i)
%If   $x_m=0$ then, since ${f_k}\vert_{z_m=0}=g_k$, $(x_0,\cdots ,x_{m-1})\not= (0,\cdots ,0)$ 
%is a non-trivial common root of $g_0,\cdots ,g_n$, which is a contradiction.
%%%%
%\par
%(ii)
%For each $0\leq k\leq n$,
%let $L_k\subset \RP^m$ and $L_k^{\p}\subset \RP^m$ be the hypersurfaces defined by
%$L_k=\{[x_0:\cdots :x_m]\in\RP^m: f_k(x_0,\cdots ,x_m)=0\}$
%and
%$L_k^{\p}=\{[x_0:\cdots :x_m]\in\RP^m: g_k(x_0,\cdots ,x_m)=0\}$.
%Since 
%$$
%\{[x_0:\cdots :x_m]\in \RP^m:x_m=0\}\cap
%\big(\cap_{k=0}^nL_k\big)=\cap_{k=0}^nL_k^{\p}=\emptyset ,
%$$
%it follows from the Projective Dimension Theorem (\cite{Hart}) that
%$V=\cap_{k=0}^nL_k$ is a finite set.
%Applying B$\hat{\mbox{e}}$zout's  Theorem [\cite{Fulton}, page 145 (3)]
%we deduce that $V$ has at most $d^m$ points.
%%%%
%\par
%(iii)
%Since the codimension of $A_d^{\C}(m,n;g)$ in $A_d^*$ is
%$2n-m+1$, $A_d(m,n;g)$ is simply connected if $m<2n$.
%Similarly, the fact that the codimension of $A_d^{\C}(m,n)$ in $A_d$ is also 
%$2n-m+1$ implies that
%$A_d^{\C}(m,n)$ is simply connected if $m<2n$.
\end{proof}
%%%%(End of Proof of Lemma 3.2)%%

%%%%%
%%(Definition )%%
\begin{defi}
%\par
%{\bf Definition. }
%%%%%%
{\rm
(i) For a finite set ${\bf x} =\{x_1,\cdots ,x_l\}\subset \R^N$,
let $\sigma ({\bf x})$ denote the convex hull spanned by ${\bf x}.$
Note that   $\sigma ({\bf x})$
is an $(l-1)$-dimensional simplex if and only if vectors
$\{x_k-x_1\}_{k=2}^l$ are linearly independent.
In particular, it is in general position if $x_1,\cdots ,x_l$ are linearly independent over $\R$.
%%%
\par
%%%
(ii) Let $h:X\to Y$ be a surjective map such that
$h^{-1}(y)$ is a finite set for any $y\in Y$, and let
$i:X\to \R^n$ be an embedding.
Let  $\mathcal{X}^{\Delta}$  and $h^{\Delta}:{\mathcal{X}}^{\Delta}\to Y$ 
denote the space and the map
defined by
%%%
$$
\mathcal{X}^{\Delta}=
\big\{(y,w)\in Y\times \R^N:
w\in \sigma (i(h^{-1}(y)))
\big\}\subset Y\times \R^N,
\ h^{\Delta}(y,w)=y.
$$
The pair $(\mathcal{X}^{\Delta},h^{\Delta})$ is called
{\it a simplicial resolution of }$(h,i)$.
In particular, $(\mathcal{X}^{\Delta},h^{\Delta})$
is called {\it a non-degenerate simplicial resolution} if for each $y\in Y$
and any $k$ points of $i(h^{-1}(y))$ span $(k-1)$-dimensional simplex of $\R^N$.
%%%%
\par
(iii)
For each $k\geq 0$, let $\mathcal{X}^{\Delta}_k\subset \mathcal{X}^{\Delta}$ be the subspace
given by 
$$
\mathcal{X}_k^{\Delta}=\big\{(y,\omega)\in \mathcal{X}^{\Delta}:
\omega\in\sigma ({\bf u}),
{\bf u}=\{u_1,\cdots ,u_l\}\subset i(h^{-1}(y)),l\leq k\big\}.
$$
%%%%
We make identification $X=\mathcal{X}^{\Delta}_1$ by identifying the point $x\in X$ with the pair
$(h(x),i(x))\in \mathcal{X}^{\Delta}_1$,
and we note that  there is an increasing filtration
$$
\emptyset =
\mathcal{X}^{\Delta}_0\subset X=\mathcal{X}^{\Delta}_1\subset \mathcal{X}^{\Delta}_2\subset
\cdots \subset \mathcal{X}^{\Delta}_k\subset \mathcal{X}^{\Delta}_{k+1}\subset
\cdots \subset \bigcup_{k= 0}^{\infty}\mathcal{X}^{\Delta}_k=\mathcal{X}^{\Delta}.
$$
}
\end{defi}

%%%(Lemma 3.4)%%
\begin{lemma}[\cite{Mo2}, \cite{Va}]\label{lemma: simp}
%%%%%%%%
Let $h:X\to Y$ be a surjective map such that
$h^{-1}(y)$ is a finite set for any $y\in Y$, and let
$i:X\to \R^N$ be an embedding.
\par
%%(i)%%
$\I$
If $X$ and $Y$ are closed semi-algebraic spaces and the
two maps $h$, $i$ are polynomial maps, then
$h^{\Delta}:\mathcal{X}^{\Delta}\stackrel{\simeq}{\rightarrow}Y$
is a homotopy equivalence.
%%%%
\par
$\II$
There is an embedding $j:X\to \R^M$ such that the associated simplicial resolution
$(\tilde{\mathcal{X}}^{\Delta},\tilde{h}^{\Delta})$
of $(h,j)$ is non-degenerate, and
the space $\tilde{\mathcal{X}}^{\Delta}$
is uniquely determined up to homeomorphism.
Moreover,
there is a filtration preserving homotopy equivalence
$q^{\Delta}:\tilde{\mathcal{X}}^{\Delta}\stackrel{\simeq}{\rightarrow}{\mathcal{X}}^{\Delta}$ such that $q^{\Delta}\vert X=\mbox{id}_X$.
\qed
%%%
\end{lemma}
%%%%%(End of Lemma 3.4)%%

%%%
\par
{\it Remark. }
Even when $h$ is not finite to one,  it is still possible to define its simplicial resolution and associated non-degenerate one.
We omit the details of this construction and refer the reader to  \cite{Mo2}.
%%%%
%\par\vspace{3mm}

%%%(Definition for Z_d for A_d(m,n;g)%%
\par
%%(Definition 3.5)%%%
%{\bf Definition. }
%%%
\begin{defi}
{\rm
Let  $Z_d\subset \Sigma_d\times \R^m$
denote %the subspace
the tautological normalization of $\Sigma_d$
consisting of all pairs 
$({\bf f},{\bf x})=((f_0,\cdots ,f_n),
(x_0,\cdots ,x_{m-1}))\in \Sigma_d\times\R^m$
such that the polynomials $f_0,\cdots ,f_n$ have a non-trivial common real root
$({\bf x},1)=(x_0,\cdots ,x_{m-1},1)$.
%%%
%\par
%%%%
Projection on the first factor  gives a surjective map
$\pi_d^{\p} :Z_d\to\Sigma_d.$
%%%
\par
%%%%%(ii)%%%
Let $\phi_d:A_d^*\stackrel{\cong}{\rightarrow}\R^{N_d^*}$
be any fixed homeomorphism, and
let ${\rm H}_d$ be the set consisting of all
monomials $\varphi_I=z^I=z_0^{i_0}z_1^{i_1}\cdots z_m^{i_m}$
of degree $d$
($I=(i_0,i_1,\cdots ,i_m)\in\Z^{m+1}_{\geq 0}$,
$\vert I\vert =\sum_{k=0}^mi_k = d$).
%%%%%%%
%\par
%%%%%%
Next, we define the  Veronese embedding, which will play a key role in our argument. 
Let
$\psi_d^* :\R^m\to \R^{M_d}$
be the map given by
$\dis
\psi_d^* (x_0,\cdots ,x_{m-1})=
\Big(\varphi_I(x_0,\cdots ,x_{m-1},1)
\Big)_{\varphi_I\in {\rm H}_d},$
where $M_d:=
\binom{d+m}{m}.$
%%%%
Now define the embedding
$\Phi_d^* :Z_d^* \to \R^{N_d^*+M_d}$ by
$$
\Phi_d^*  ((f_0,\cdots ,f_n),{\bf x})=
(\phi_d^* (f_0,\cdots ,f_n),\psi_d^* ({\bf x})).
$$
%where  
%$L_d^* :=N_d^* +M_d $.
%%%%%
}
\end{defi}

%%%(new)%%%
%%%%(Definition of simplicial resolutions)%%
%\par\vspace{2mm}
%
%%(Definition 3.6)%%
%{\bf Definition. }
%%%
\begin{defi}
{\rm
Let 
$(\SZ(d),\ ^{\p}{\pi_d^{\Delta}}:\SZ(d)\to\Sigma_d)$ 
and
$(\tilde{\SZ}(d),\ ^{\p}{\tilde{\pi}_d^{\Delta}}:\tilde{\SZ}(d)\to\Sigma_d)$ 
denote the simplicial resolution of
$(\pi_d^{\p},\Phi_d^*)$ 
and the corresponding non-degenerate
simplicial resolution
with the natural increasing filtrations
$$
\begin{cases}
\dis
\SZ(d)_0=\emptyset
\subset \SZ(d)_1\subset 
\SZ(d)_2\subset \cdots
\subset 
\SZ(d)=\bigcup_{k= 0}^{\infty}\SZ (d)_k,
\\
\dis
\tilde{\mathcal{Z}}^{\Delta}(d)_0
=\emptyset
\subset \tilde{\mathcal{Z}}^{\Delta}(d)_1
\subset
\tilde{\mathcal{Z}}^{\Delta}(d)_2\subset \cdots
\subset \tilde{\mathcal{Z}}^{\Delta}(d)=\bigcup_{k=0}^{\infty}
\tilde{\mathcal{Z}}^{\Delta}(d)_k.
\end{cases}
$$
}
\end{defi}
%%%%%%%%%
By Lemma \ref{lemma: simp}, the map
$^{\p}{\pi_d^{\Delta}}:
\SZ(d)\stackrel{\simeq}{\rightarrow}\Sigma_d$ is a homotopy equivalence. It is easy to see that it
extends to a homotopy equivalence
$^{\p}{\pi_{d+}^{\Delta}}:\SZ(d)_+\stackrel{\simeq}{\rightarrow}{\Sigma_{d+}},$
%%%%%
where $X_+$ denotes the one-point compactification of a
locally compact space $X$.
\par
Since
${\SZ (d)_r}_+/{\SZ (d)_{r-1}}_+
\cong (\SZ (d)_r\setminus \SZ (d)_{r-1})_+$,
we have the  Vassiliev type spectral sequence
%%%%%%%%%%%
$$
\big\{E_t^{r,s}(d),
d_t:E_t^{r,s}(d)\to E_t^{r+t,s+1-t}(d)
\big\}
\Rightarrow
H^{r+s}_c(\Sigma_d,\Z),
$$
where
$H_c^k(X,\Z)$ denotes the cohomology group with compact supports given by 
$H_c^k(X,\Z):= H^k(X_+,\Z)$ and
$E_1^{r,s}(d):=H^{r+s}_c(\SZ(d)_r\setminus\SZ(d)_{r-1},\Z)$.
%%%%%%%%%%%%%
\par
%%%%%

It follows from the Alexander duality that there is a natural
isomorphism
%%%(8)%%%
\begin{equation}\label{Al}
H_k(A_d^{\C}(m,n;g),\Z)\cong
H_c^{N_d^*-k-1}(\Sigma_d,\Z)
\quad
\mbox{for }1\leq k\leq N_d^*-2.
\end{equation}
%%%
Using (\ref{Al}) and
reindexing we obtain a
spectral sequence
%%
%%
%%%(9)%%%
\begin{eqnarray}\label{SS}
%%%%%%%%%%%%%%%%%%%
&&\big\{\E^t_{r,s}(d), \tilde{d}^t:\E^t_{r,s}(d)\to \E^t_{r+t,s+t-1}(d)\big\}
\Rightarrow H_{s-r}(A_d^{\C}(m,n;g),\Z)
\end{eqnarray}
%%%%%%%
if $s-r\leq N_d^*-2$,
where
$\E^1_{r,s}(d)=
H^{N_d^*+r-s-1}_c(\SZ(d)_r\setminus\SZ(d)_{r-1},\Z).$
%and $\tilde{E}^t_{r,s}(d)=E_t^{r,N_d^*-1-s}(d).$
%%%%%%

%%%%%%%(Lemma 3.7)%%%%%%%
\begin{lemma}\label{lemma: simplex*}
%%%%%%%
$\I$
If
$\{y_1,\cdots ,y_{r}\}\in C_r(\R^{m})$ is  any set of $r$ distinct points
in $\R^m$ and $r\leq d+1$, then the $r$ vectors $\{\psi_d^*(y_k):1\leq k\leq r\}$
are linearly independent
over $\R$ and
span an $(r-1)$-dimensional simplex in $\R^{M_d}$.
\par
$\II$
If $1\leq r\leq d+1$,
there is a homeomorphism
$$
\SZ (d)_r\setminus \SZ (d)_{r-1}\cong
\tilde{\SZ} (d)_r\setminus \tilde{\SZ} (d)_{r-1}
$$
%%%%%
\end{lemma}
%%%%%%%
\begin{proof}
%%%%%%%
The proof is completely analogous to that of
[\cite{AKY1}, Lemma 4.3].
%Writing $y_k=(y_{0,k},\cdots ,y_{m-1,k})$ for 
%$1\leq k\leq  r$, for each $i\not= j$ 
%we can find a number $l$ $(0\leq l\leq m-1)$
%such that $y_{l,i}\not= y_{l,j}$. By a linear change of coordinates we can ensure that $l=0$ for all $i,j$.
%%%%
%The assertion (i) follows form the fact that the Vandermonde matrix constructed from the powers $y_{0,i}$ is non-singular. 
%The assertion (ii) also easily follows from (i).
\end{proof}
%%%(End of proof of Lemma 3.7)%%%%

%%%%(Lemma 3.8)%%%%
\begin{lemma}\label{lemma: vector bundle*}
%%%%%%%%%%%%%%%%%%
%%%
If $1\leq r\leq \lfloor\frac{d+1}{2}\rfloor$,
$\SZ(d)_r\setminus\SZ(d)_{r-1}$
is homeomorphic to the total space of a real
vector bundle $\xi_{d,r}$ over $C_r(\R^m)$ with rank 
$l_{d,r}^*=N_d^*-(2n+1)r-1$.
%%%%%%%%%%%%%%%%%%
\end{lemma}
%%%%%%%%(Proof of Lemma 3.8)%%%
\begin{proof}
%%%%%%%%%%%%%%%%%%
The proof is completely analogous to that of [\cite{AKY1}, Lemma 4.4].
\end{proof}
%%%%(End of proof of Lemma 3.8)%%%%%
%%

%%(Lemma 3.9)%%%%%%%%
\begin{lemma}\label{lemma: range**}
%%%%%%%%%
All non-zero entries of $\E^1_{r,s}(d)$ are
situated in the range
%%%%
$s\geq r\big(2n+2-m\big)$ if $1\leq r\leq d^m$,
and $\E^1_{r,*}(d)=0$ if $r>d^m$.
%%%
\end{lemma}
%%%
\begin{proof}
%%%%%%
The proof is completely analogous to that of
[\cite{AKY1}, Lemma 4.5].
\end{proof}
%%%(End of Lemma 3.9)%%%
%
%%%%%%
%%%%%%(Lemma 3.10)%%
\begin{lemma}\label{lemma: range*}
If $1\leq r\leq \lfloor \frac{d+1}{2}\rfloor$, there is a natural isomorphism
$$
\E^1_{r,s}(d)\cong
%H_c^{(n+1)r-s}(C_r(X_K),{\cal O}_r)
H_{s-(2n-m+2)r}(C_r(\R^m),(\pm \Z)^{\otimes (2n-m+1)}).
$$
\end{lemma}
%\begin{proof}
%%%(Proof of Lemma 3.10)%%
\begin{proof}
%%%%%%
By using the Thom isomorphism and Poincar\'e duality,
we obtain the desired isomorphism.
\end{proof}
%%%%(End of proof of Lemma 3.10)%%%
\par\vspace{3mm}\par

%%%%%%%%%%%%%%%%%%%%%%%%%%%%%%%%%
Now we recall the spectral sequence constructed by V. Vassiliev
\cite{Va}.
From now on, we will assume that $m\leq 2n$, and
% and $X$ is 
%a finite $m$ dimensional simplicial complex such that
%$X$ is $C^{\infty}$-imbedded in $\R^L$.
%%
%Considering $S^{2n+1}$ 
%and $X$ as subspaces $S^{2n+1}\subset \R^{2n+2},  X\subset \R^L$, 
we identify $\Map (S^m,S^{2n+1})$ with the space
$\Map (S^m,\R^{2n+2}\setminus \{{\bf 0}_{2n+2}\})$.
We also choose  a map
$\varphi :S^m\to \R^{2n+2}\setminus\{{\bf 0}_{2n+2}\}$
and fix it.
%%
%%%%%
Observe that $\Map (S^m,\R^{2n+2})$ is a linear space
and consider
the complements
%$$
%\begin{cases}
${\frak A}_m^n=\Map (S^m,\R^{2n+2})\setminus \Map (S^m,S^{2n+1})$
and
%\\
$\tilde{\frak A}_m^n=\Map^* (S^m,\R^{2n+2})\setminus \Map^*(S^m,S^{2n+1}).$
%\end{cases}
%$$
\par
Note that  ${\frak A}_m^n$ consists of all continuous maps
$f:S^m\to \R^{2n+2}$ passing through ${\bf 0}_{2n+2}.$
%%%%
We will denote by $\Theta^k_{\varphi}\subset \Map (S^m,\R^{2n+2})$ the subspace
consisting of all maps $f$ of the forms $f=\varphi +p$,
where $p$ is the restriction to $S^m$ of a polynomial map
$\R^{m+1}\to \R^{2n+2}$ of degree $\leq k$.
Let
$\Theta^k\subset \Theta^k_{\varphi}$ denote the subspace
consisting of all $f\in \Theta_{\varphi}^k$
passing through ${\bf 0}_{n+1}.$ 
%which
% intersects ${\bf 0}_{n+1}$.
In [\cite{Va}, page 111-112] Vassiliev uses the space 
$\Theta^k$ as a finite dimensional approximation of ${\frak A}^n_m$.%\footnote
%%%%
%In fact,
%by using the Alexander duality, for any $N\geq 1$
%there is an isomorphism $H_N(\Map (S^m,S^n),\Z)\cong H^{M_k^*-N}_c(\Theta^k,\Z)$
%if $k$ is a sufficiently large degree,
% where $M_k^*=\dim \Theta_{\varphi}^k=\sum_{j=0}^k\binom{j+m}{m}$.
%\footnote
%{Note that the proof 
% of this fact given by Vassiliev, makes use of the Stone-Weierstrass theorem, 
% so, although we are now not using the stable result of Section 2, something like it is also implicitly  involved here. }
%%%%
\par
%%%
Let $\tilde{\Theta}^k$ denote the subspace of $\Theta^k$ consisting of all maps 
$f\in \Theta^k$ which preserve
the base points.
%%%
By a variation of the preceding argument, Vassiliev also shows that  
$\tilde{\Theta}^k$ can be used as a finite dimensional approximation of  
$\tilde{\frak A}_m^n$ [\cite{Va}, page 112].
\par
%%%%%%%%
%%
%By choosing a sufficiently fine  polynomial approximation of $\tilde{\frak A}_m^n$, 
Let $\mathcal{X}_k\subset \tilde{\Theta}^k\times \R^{m+1}$ denote the %tautological normalization of $\tilde{\Theta}^k$
subspace consisting of all pairs 
$(f,\alpha)\in\tilde{\Theta}^k \times \R^{m+1}$
such that $f(\alpha )={\bf 0}_{2n+2}$, and
let $p_k:\mathcal{X}_k\to \tilde{\Theta}^k$ be the projection onto the first factor.
Then,  by making use of simplicial resolutions of the surjective maps 
$\{p_k:k\geq 1\}$,
one can construct an associated geometric resolution 
$\{\tilde{\frak A}_m^n\}$ of $\tilde{\frak A}_m^n,$ 
whose cohomology approximates the homology of $\Map^*(S^m,S^{2n+1})=\Omega^mS^{2n+1}$ to any desired dimension. 
%%%%
From the natural filtration on the approximating space
%%%()%%
%\begin{equation*}\label{geometric resolution}
%%%%%%%
$\dis F_1\subset F_2\subset F_3\subset 
\cdots \subset \bigcup_{k=1}^{\infty}F_k=\{\tilde{\frak A}_m^n\},$
%\end{equation*}
%%%
we obtain the associated spectral sequence:
%%
%%
%%%(10)%%%
\begin{equation}\label{SSS}
%%%%%%%
\{E^t_{r,s},d^t:
E^t_{r,s}\to
E^t_{r+t,s+t-1}\}
\Rightarrow
H_{s-r}(\Omega^mS^{2n+1},\Z).
\end{equation}
%%%
The following result follows easily from
[\cite{Va}, Theorem 2 (page 112) and (32) (page 114)].
%%%%%%%%%%%%%%%%%%%%%%

%%%(Lemma 3.11: Vassiliev's result)%%%
\begin{lemma}[\cite{Va}]
\label{lemma: Va}
%%%%%%%%%%%%%%%%
Let $2\leq m\leq 2n$ be integers and let
$X$ be a finite $m$-dimensional simplicial complex
with a fixed base point $x_0\in X$.
\begin{enumerate}
%%(i)%
\item[$\I$]
$E^1_{r,s}=
H_{s-(2n-m+2)r}(C_r(\R^m),(\pm \Z)^{\otimes (2n-m+1)})$ if $r\geq 1$,
and $E^1_{r,s}=0$ if $r<0$ or $s<0$ or $s<(2n-m+2)r$.
%%%(ii)%%
\item[$\II$]
For any $t\geq 1$,
$d^t=0:
E^t_{r,s}\to
E^t_{r+t,s+t-1}$ for all $(r,s)$,
and $E^1_{r,s}=E^{\infty}_{r,s}$.
%%%
Moreover,
for any $k\geq 1$, the extension problem for
the graded group
$Gr (H_k(\Omega^mS^{2n+1},\Z))=
\bigoplus_{r=1}^{\infty}E^{\infty}_{r,r+k}
=\bigoplus_{r=1}^{\infty}E^{1}_{r,r+k}$
is trivial and 
there is an isomorphism
%\begin{eqnarray*}
$$
H_{k}(\Omega^mS^{2n+1},\Z)
%&\cong&
% \bigoplus_{r=1}^{\infty}E^1_{r,k+r}=
%  \bigoplus_{r=1}^{\infty}E^{\infty}_{r,k+r}
%  \\ &=&
\cong
\bigoplus_{r=1}^{\infty}
H_{k-(2n-m+1)r}(C_r(\R^m),(\pm \Z)^{\otimes (2n-m+1)}).
\qed
$$
%\end{eqnarray*}
\end{enumerate}
\end{lemma}
%%
%%%%(End of Lemma 3.11)%%%

%\par
%{\bf Definition. }
%%(Definition 3.12)%%
\begin{defi}
{\rm
We identify $\Omega^mS^{2n+1}=\Omega^m(\C^{n+1}\setminus \{{\bf 0}\})$ and
define the map $j_d^{\p}:A_d^{\C}(m,n;g)\to \Omega^mS^{2n+1}$ by
$$
j_d^{\p}(f_0,\cdots ,f_n)(x_0,\cdots ,x_m)=
(f_0(x_0,\cdots ,x_m),\cdots\cdots ,f_n(x_0,\cdots ,x_m))
$$
for $((f_0,\cdots ,f_n),(x_0,\cdots ,x_m))\in A_d^{\C}(m,n;g)\times S^m$.
}
%%%
\end{defi}

Now, by applying the spectral sequence
(\ref{SSS}),  we prove the following result,  which plays a key role the proof of Theorem \ref{thm: I} .

%%%(Theorem 3.13)%%%
\begin{thm}\label{thm: I*}
%%%
Let $m,n\geq 2$ be positive integers such that $2\leq m\leq 2n$, and 
let $g\in \Alg_d^*(\RP^{m-1},\CP^n)$ be an algebraic map of minimal degree $d$.
%%%
\begin{enumerate}
%%(i)%%
\item[$\I$]
%%%
The map
$j_{d}^{\p}:A_d^{\C}(m,n;g)\to  \Omega^mS^{2n+1}$
is a homotopy equivalence through dimension $D_{\C}(d;m,n)$ if
$m<2n$ and a homology equivalence through dimension $D_{\C}(d;m,n)$
if $m=2n$. 
%%(ii)%%%
\item[$\II$]
For any $k\geq 1$, 
$H_k(A_d^{\C}(m,n;g),\Z)$ contains the subgroup
$$
G^d_{m,2n+1}=\bigoplus_{r=1}^{\lfloor \frac{d+1}{2}\rfloor}
H_{k-(2n-m+1)r}(C_r(\R^m),(\pm \Z)^{\otimes (2n-m+1)} )
$$
as a direct summand.
%%%%%%
\end{enumerate}
\end{thm}

%%(Proof of Theorem 3.13)%%%
\begin{proof}
%%%%%%%%%%%%%%
Consider the spectral sequence (\ref{SSS}).
First, note that, by Lemma \ref{lemma: simp}, there is a filtration
preserving homotopy equivalence 
$^{\p}q^{\Delta}:\tilde{\SZ} (d)\stackrel{\simeq}{\rightarrow}\SZ (d).$
Note also that the image of the  map $j_d^{\p}$ lies 
in a space of polynomial mappings, which approximates  the space of continuous mappings
$S^m\to S^{2n+1}$.
Since $\tilde{\mathcal{Z}}^{\Delta}(d)$ is non-degenerate,
the map$j_d^{\p}$ naturally extends to a filtration
preserving map 
$^{\p}{\tilde{\pi}}: \tilde{\SZ} (d)\to \{\tilde{\frak A}_m^n\}$
between  resolutions.
Thus the filtration preserving maps
$$
\begin{CD}
\SZ (d) @<^{\p}q^{\Delta}<\simeq< \tilde{\SZ} (d)
@>^{\p}\tilde{\pi}>> \{\tilde{\frak A}_m^n\}
\end{CD}
$$
induce a homomorphism of spectral sequences
$\{\tilde{\theta}^t_{r,s}:\E^t_{r,s}(d)\to E^t_{r,s}\},$
where $\{E^t_{r,s},d^t\}\Rightarrow H_{s-r}(\Omega^mS^n,\Z).$
%%%%
\par
%%%
Observe that,
by Lemma \ref{lemma: range*}, 
(ii) of Lemma \ref{lemma: Va} and
the naturality of Thom isomorphism,
for $r\leq \lfloor \frac{d+1}{2}\rfloor$
 there is a
commutative diagram
%%%(11)%%
\begin{equation}\label{Thom2}
%%%%%%%
\begin{CD}
\E^1_{r,s}(d) 
@>T>\cong> H_{s-r(2n-m+2)}(C_r(\R^m),(\pm \Z)^{\otimes (2n-m+1)})
\\
@V\tilde{\theta}^1_{r,s}VV  \Vert @.
\\
E^1_{r,s} 
@>T>\cong> H_{s-r(2n-m+2)}(C_r(\R^m),(\pm \Z)^{\otimes (2n-m+1)} )
\end{CD}
\end{equation}
%%%%%%%%
%for $r\leq \lfloor \frac{d+1}{2}\rfloor$.
Hence, if $r\leq \lfloor \frac{d+1}{2}\rfloor$,
$\tilde{\theta}^1_{r,s}:\E^1_{r,s}(d)\stackrel{\cong}{\rightarrow}
E^1_{r,s}$ and thus so is
$\tilde{\theta}^{\infty}_{r,s}:\E^{\infty}_{r,s}(d)\stackrel{\cong}{\rightarrow}
E^{\infty}_{r,s}$.
%%%%%%%%%%%%%%%%%%%%%
\par
%%%
Next, we will compute the number
$$
D_{min}=\min\{
N\vert \ N\geq s-r,\ 
s\geq (2n+2-m)r,\ 
1\leq r<\lfloor \frac{d+1}{2}\rfloor +1
\}.
$$
%%%%
It is easy to see that  $D_{min}$ is the
largest integer $N$ which satisfies the inequality
%%%%%%
$(n+1-m)r> r+N$
for $r=\lfloor\frac{d+1}{2}\rfloor +1$, hence
%%(12)%%
\begin{equation}\label{Dmin}
%%%%%%
D_{min}=
(2n-m+1)(\lfloor\frac{d+1}{2}\rfloor +1)-1=D_{\C}(d;m,n).
\end{equation}
%%%%%%%%%%%
We note that, for dimensional reasons,
$\tilde{\theta}^{\infty}_{r,s}:
\E^{\infty}_{r,s}(d)\stackrel{\cong}{\rightarrow} 
E^{\infty}_{r,s}$ is always
an isomorphism if
$r\leq \lfloor \frac{d+1}{2}\rfloor$ and $s-r\leq D_{\C}(d;m,n)$.
\par
On the other hand,  from Lemma \ref{lemma: range**}
it easily follows that
 $\tilde{E}^1_{r,s}(d)=E^1_{r,s}=0$ if
$s-r\leq D_{\C}(d;m,n)$ and $r>\lfloor \frac{d+1}{2}\rfloor$.
Hence,
we have:
%%%(12.1)%%%%%
\begin{enumerate}
\item[(\ref{Dmin}.1)]
If  $s\leq r+D_{\C}(d;m,n)$, then
$\tilde{\theta}^{\infty}_{r,s}:
\E^{\infty}_{r,s}(d)\stackrel{\cong}{\rightarrow} 
E^{\infty}_{r,s}$ is
always an isomorphism.
\end{enumerate}
%%%
Hence, by using the comparison Theorem of spectral sequences, 
we have that 
$j_{d}^{\p}$ 
%${i_d^{\p}}_*:H_k(\Alg^*_d(m,n;g),\Z)
%\stackrel{\cong}{\rightarrow} H_k(\Omega^mS^n,\Z)$
is a homology equivalence through dimension $D_{\C}(d;m,n)$.
%%
%\par
%%
Since $A_d^{\C}(m,n;g)$ and $\Omega^mS^{2n+1}$ are simply connected
if $m<2n$,
$j_{d}^{\p}$ is a homotopy equivalence through dimension $D_{\C}(d;m,n)$
if  $m<2n$.
Hence, (i) is proved.
%%%%%%
\par
It remains to show (ii).
Since
$d^t=0$ for any $t\geq 1$, from
the equality
$d^t\circ \tilde{\theta}^t_{r,s}
=\tilde{\theta}^t_{r+t,s+t-1}\circ \tilde{d}^t$
and some diagram chasing,
we obtain $\E^1_{r,s}(d)=\E^{\infty}_{r,s}(d)$ for all $r\leq \lfloor \frac{d+1}{2}\rfloor$.
%%%%
Moreover, since
the extension problem for the graded group
$$
Gr (H_k(\Omega^mS^{2n+1},\Z))=\bigoplus_{r=1}^{\infty}E^{\infty}_{r,k+r}
=\bigoplus_{r=1}^{\infty}E^{1}_{r,k+r}
$$
is trivial,  by using (\ref{Dmin}.1) we  can  prove
that the associated graded group
$$
Gr (H_k(A_d^{\C}(m,n;g),\Z))=\bigoplus_{r=1}^{\infty}\E^{\infty}_{r,k+r}(d)
=\bigoplus_{r=1}^{\infty}\E^{1}_{r,k+r}(d)
$$
is also trivial until the $\lfloor\frac{d+1}{2}\rfloor$-th term of the filtration.
Hence,  $H_k(A_d^{\C}(m,n;g),\Z)$ contains the subgroup
$$
\bigoplus_{r=1}^{\lfloor \frac{d+1}{2}\rfloor}\E^1_{r,k+r}(d)
=
\bigoplus_{r=1}^{\lfloor \frac{d+1}{2}\rfloor}\E^{\infty}_{r,k+r}(d)
\cong
\bigoplus_{r=1}^{\lfloor \frac{d+1}{2}\rfloor}
H_{k-r(2n-m+1)}(C_r(\R^m),(\pm \Z)^{\otimes (2n-m+1)})
$$
as a direct summand, which proves the assertion (ii).
\end{proof}
%%%%(End of Proof of Theorem 3.13)%%

%%%(Corollary 3.14)%%
\begin{cor}\label{cor: I*}
%%%%%
Let $m,n\geq 2$ be positive integers such that $2\leq m\leq 2n$,  
let $g\in \Alg_d^*(\RP^{m-1},\CP^n)$ be an algebraic map of minimal degree $d$,
and let $\Bbb F =\Z/p$ $($p: prime$)$ or $\Bbb F =\Bbb Q$.
%%%
Then the map $j_d^{\p}:A_d^{\C}(m,n;g)\to \Omega^mS^{2n+1}$ induces an isomorphism on
the homology group $H_k(\ ,\Bbb F)$ for any $1\leq k\leq D_{\C}(d;m,n)$.
%%%
\end{cor}
%%%%
\begin{proof}
In the  proof of Theorem \ref{thm: I*} replace the homology groups $H_k(\ ,\Z)$ and $H_k(\ ,(\pm \Z)^{\otimes k})$ by $H_k(\ ,\Bbb F)$ 
and $H_k(\ ,(\pm \Bbb F)^{\otimes k})$ and use the same argument.
% as that of Theorem \ref{thm: II**}.
\end{proof}
%%%%%(End of proof of Corollary 3.14)%%%

\par
Let $\gamma_m:S^{m}\to \RP^m$ and $\gamma_n^{\C}:S^{2n+1}\to\CP^n$
denote the usual double covering and the
Hopf fibration map, respectively.  
Let 
$\gamma_m^{\#}:\Map^*_{\epsilon}(\RP^m,\CP^n)\to \Omega^m\CP^n$ be given by
$\gamma_m^{\#}(h)=h\circ \gamma_m$.
It is easy to verify that the following  diagram
%%
%%(13)%%
\begin{equation}\label{gamma}
%%%%
\begin{CD}
A_d^{\C}(m,n;g) @>\Psi_d^{\p}>\cong> \Alg_d^{\C}(m,n;g) @>i_d^{\p}>\subset> F_d(m,n;g)
\\
@V{j_d^{\p}}VV @. @V{i^{\p}}V{\cap}V
\\
\Omega^mS^{2n+1} @>\Omega^m\gamma_n^{\C}>\simeq> 
\Omega^m\CP^n @<\gamma_m^{\#}<< \Map^*_{[d]_2}(\RP^m,\CP^n)
\end{CD}
\end{equation}
%%%%%
is commutative,
where $i^{\p}:F_d(m,n;g)\stackrel{\subset}{\rightarrow} \Map_{[d]_2}^*(\RP^m,\CP^n)$ 
and $\Psi_d^{\p}$ denote the inclusion and the restriction
$\Psi_d^{\p}=\Psi_d^{\C}\vert A_d^{\C}(m,n;g)$,
respectively.
%%%
%%
%%%(Lemma 3.15)%%
\begin{lemma}\label{lemma: II}
%%%
If $2\leq m<2n$  and
$g\in\Alg_d^*(\RP^{m-1},\RP^n)$ a fixed map of minimal degree $d$,
then the map
$\gamma_m^{\#}\circ i^{\p}:F_d(m,n;g)\to\Omega^m\CP^n$
is a homotopy equivalence through dimension $D_{\C}(d;m,n)$.
\end{lemma}
%%%%%%
\begin{proof}
%%%%
Since there is a homotopy equivalence $F_d(m,n;g)\simeq \Omega^m\CP^n$,
the two spaces $F_d(m,n;g)$ and $\Omega^m\CP^n$ are simple.
So it suffices to show that the map $\gamma_m^{\#}\circ i^{\p}$ is a homology equivalence
through dimension $D_{\C}(d;m,n)$.
\par
Let $\Bbb F =\Z/p$ $(p$: prime) or $\Bbb F =\Bbb Q$, and consider the induced homomorphism
$(\gamma_m^{\#}\circ i^{\p})_*=H_k(\gamma_m^{\#}\circ i^{\p},\Bbb F)
:H_k(F_d(m,n;g),\Bbb F)\to H_k(\Omega^m\CP^n,\Bbb F)$.
%%%
Since $\Omega^m\gamma_n^{\C}$ is a homotopy equivalence by
 Corollary \ref{cor: I*} and the commutativity of the diagram (\ref{gamma})
$(\gamma_m^{\#}\circ i^{\p})_*$ is an epimorphism for any $1\leq k\leq D_{\C}(d;m,n)$.
\par
However, since there is a homotopy equivalence
$F_d(m,n;g)\simeq \Omega^m\CP^n$, we have
$\dim_{\Bbb F}H_k(F_d(m,n;g),\Bbb F)=\dim_{\Bbb F}H_k(\Omega^m\CP^n,\Bbb F)<\infty$
for any $k$.
Hence,
$H_k(\gamma_m^{\#}\circ i^{\p},\Bbb F )$ is an isomorphism for any $1\leq k\leq D_{\C}(d;m,n)$.
By the Universal Coefficient Theorem $\gamma_m^{\#}\circ i^{\p}$ induces an isomorphism on $H_k(\ ,\Z)$ for any $1\leq k\leq D_{\C}(d;m,n)$.
\end{proof}
%%(End of Proof of Lemma 3.15)%%%%

%%(Proof of Theorem 1.4)%%%
\begin{proof}[Proof of Theorem \ref{thm: I}]
%%%%%
Since $\Psi_d^{\p}:A_d^{\C}(m,n;g)\stackrel{\cong}{\rightarrow}\Alg_d^{\C}(m,n;g)$ is a
homeomorphism, the assertion (ii) follows from Theorem \ref{thm: I*}.
%%%
Because $\gamma_m^{\#}\circ i^{\p}$ a homotopy equivalence through dimension $D(d;m,n)$,
by using the  diagram (\ref{gamma}) and Theorem \ref{thm: I*}
 we easily obtain the assertion (i).
%%%
\end{proof}
%%%(End of Proof of Theorem 1.6)%%%

%%%(SECTION 4)%%%
\section{The space $A_d^{\C}(m,n)$.}\label{section 3.3}
In this section we shall consider the unstable problem for the space
$A_d^{\C}(m,n)$, where $d = 2 d^*\geq 2$ is even.

%\par\vspace{2mm}
%%%
%%%(Definition 4.1)%%%%
\begin{defi}
%\par
%{\bf Definition. }
%%%%%%%%%%%%%%%%%%%%%%%%%%%
{\rm
Define
$\psi_{m,n}:A_d^{\R}(m,2n+1)\to A_d^{\C}(m,n)$ by
$$
\psi_{m,n}(f_0,\cdots ,f_{2n+1})=
(f_0+\sqrt{-1}f_1,f_2+\sqrt{-1}f_3,\cdots ,
f_{2n}+\sqrt{-1}f_{2n+1}).
$$
}
\end{defi}
%%%
It is easy to see that
%%%
%%(Lemma 4.2)%%%
\begin{lemma}\label{lemma: varphi}
%%%%
$\psi_{m,n}:A_d^{\R}(m,2n+1)
\stackrel{\cong}{\rightarrow}
A_d^{\C}(m,n)$ is a homeomorphism.
\qed
\end{lemma}
%%%

%%%(Lemma 4.3)%%%
\begin{lemma}\label{lemma: simply}
%%%%
$\I$
$\Map^*(\RP^m,S^{2n+1})$ is $(2n-m)$-connected.
\par
%%(ii)%%
$\II$
$\pi_{2n-m+1}(\Map^*(\RP^m,S^{2n+1}))\cong
\begin{cases}
\Z & \mbox{ if }\ m\equiv 1\ \mo ,
\\
\Z/2 & \mbox{ if }\ m\equiv 0\  \mo .
\end{cases}
$
\end{lemma}
%%%%
\begin{proof}
%%%
(i)
We argue by induction on $m$.
For $m=1$ the result follows from the homotopy equivalence
$\Map^*(\RP^1,S^{2n+1})\simeq
\Omega S^{2n+1}$ .
Suppose that the space
$\Map^*(\RP^{m-1},S^{2n+1})$
is $(2n-m+1)$-connected for some $m\geq 2$.
Since $\Omega^mS^{2n+1}$ is $(2n-m)$-connected,
from the restriction fibration sequence
$\Omega^mS^{2n+1} \to \Map^*(\RP^m,S^{2n+1}) \stackrel{}{\rightarrow}
\Map^*(\RP^{m-1},S^{2n+1})$,
we deduce that $\Map^*(\RP^m,S^{2n+1})$ is $(2n-m)$-connected.
Hence, (i) has been proved.
\par
%%(ii)%%
(ii) First, consider the case $m=1$. Since $\Map^*(\RP^1,S^{2n+1})\simeq
\Omega S^{2n+1}$, (ii) clearly holds for $m=1$.
Next, consider the case $m=2$.
If we consider the fibration sequence
$\Map^*(\RP^2,S^{2n+1})\to \Omega S^{2n+1}\stackrel{2}{\rightarrow}\Omega S^{2n+1}$
induced from
the cofibration sequence $S^1\stackrel{2\iota_1}{\rightarrow}S^1 \to \RP^2$,
an easy computation shows that $\pi_{2n-1}(\Map^*(\RP^2,S^{2n+1}))\cong \Z/2.$
Hence, (ii) holds for $m=2$, too.
\par
Now we assume that $m\geq 3$ and consider the fibration sequence
$$
\Map^*(\RP^m/\RP^{m-2},S^{2n+1})\to
\Map^*(\RP^m,S^{2n+1})\stackrel{res}{\longrightarrow}
\Map^*(\RP^{m-2},S^{2n+1})
$$
%induced from the cofibration sequence
%$\RP^{m-2}\stackrel{\subset}{\rightarrow} \RP^m \to \RP^m/\RP^{m-2}=\P^m_{m-1}$.
Since $\Map^*(\RP^{m-2},S^{2n+1})$ is $(2n-m+2)$-connected, 
there is an isomorphism
$\pi_{2n-m+1}(\Map^*(\RP^m,S^{2n+1}))\cong
\pi_{2n-m+1}(\Map^*(\RP^m/\RP^{m-2},S^{2n+1}))$.
Thus it remains to show the following:
%%(15)%%
\begin{equation}\label{picomp}
\pi_{2n-m+1}(\Map^*(\RP^m/\RP^{m-2},S^{2n+1}))
\cong
\begin{cases}
\Z & \mbox{ if }\ m\equiv 1\ \mo ,
\\
\Z/2 & \mbox{ if }\ m\equiv 0\  \mo .
\end{cases}
\end{equation}
%%%
If $m\equiv 1$ $\mo$, $\RP^m/\RP^{m-2}=S^{m-1}\vee S^m$ and
there is an isomorphism
%a homotopy equivalence
%$\Map^*(\RP^m/\RP^{m-2},S^{2n+1})\simeq
%\Omega^{m-1}S^{2n+1}\times\Omega^{m}S^{2n+1}$.
%Hence,
$$
\pi_{2n-m+1}(\Map^*(\RP^m/\RP^{m-2},S^{2n+1}))\cong
\pi_{2n-m+1}(\Omega^{m-1}S^{2n+1}\times
\Omega^mS^{2n+1})\cong \Z .
$$
Hence (\ref{picomp}) holds for $m\equiv 1$ $\mo$.
%%%
Finally suppose that $m\equiv 0$ $\mo$.
Then, because
 $\RP^m/\RP^{m-2}=S^{m-1}\cup_2e^m$,
there is a fibration sequence
$$
\Map^*(\RP^m/\RP^{m-2},S^{2n+1})\to \Omega^{m-1}S^{2n+1}
\stackrel{2}{\rightarrow}
\Omega^{m-1}S^{2n+1}.
$$
From  the homotopy exact sequence induced by the above sequence,
we deduce that 
$\pi_{2n-m+1}(\Map^*(\RP^m/\RP^{m-2},S^{2n+1}))\cong \Z/2$.
\end{proof}
%%(End of proof of Lemma 4.3)%%

%(Definition 4.4)%%%%%%%%%
%\par
%{\bf Definition. }
\begin{defi}
%%%%%%%%%%%%%%%
{\rm
(i)
Let $\gamma_n:S^n\to \RP^n$ and
$\gamma_n^{\C}:S^{2n+1}\to \CP^n$ denote the usual double covering
and the Hopf fibering as before.
Define the following two maps
$$
\begin{cases}
{\gamma_n}_{\#}:\Map^*(\RP^m,S^n)\to\Map^*(\RP^m,\RP^n)
\\
{\gamma_n^{\C}}_{\#}:\Map^*(\RP^m,S^{2n+1})\to
\Map^*(\RP^m,\CP^n)
\end{cases}
$$
by ${\gamma_n}_{\#}(h)=\gamma_n\circ h$ and
${\gamma_n^{\C}}_{\#}(h^{\p})=\gamma_n^{\C}\circ h^{\p}$.
\par
Since $\Map^*(\RP^m,S^n)$ and $\Map^*(\RP^m,S^{2n+1})$) contain the subspace of constant maps,
the images of ${\gamma_n}_{\#}$ and that of ${\gamma_n^{\C}}_{\#}$ are contained in 
$\Map_0^*(\RP^m,\RP^n)$ and
$\Map^*_0(\RP^m,\CP^n)$, respectively.
Thus we obtain two maps
%%()%%
$$
\begin{cases}
{\gamma_n}_{\#}:\Map^*(\RP^m,S^n)\to \Map^*_0(\RP^m,\RP^n),
\\
{\gamma_n^{\C}}_{\#}:\Map^*(\RP^m,S^{2n+1}) \to \Map^*_0(\RP^m,\CP^n).
\end{cases}
$$
\par
(ii) Let $\mu_n:\RP^{2n+1}\to \CP^n$ denote the usual projection given by
$$
\mu_n([x_0:x_1:\cdots :x_{2n+1}]=[x_0+\sqrt{-1}x_1:
\cdots :x_{2n}+\sqrt{-1}x_{2n+1}],
$$
and define the map
${\mu_n}_{\#}:\Map_0^*(\RP^m,\RP^{2n+1})\to\Map_0^*(\RP^m,\CP^n)$ by
${\mu_n}_{\#}(h)=\mu_n\circ h$.
%%%
}
\end{defi}

%%%(Lemma 4.5)%%
\begin{lemma}\label{lemma: Hopf}
%%%
$\I$
If $1\leq m<n$,
${\gamma_n}_{\#}:\Map^*(\RP^m,S^n)\stackrel{\simeq}{\rightarrow} \Map^*_0(\RP^m,\RP^n)$
is a homotopy equivalence.
\par
$\II$
If $2\leq m\leq 2n$, 
${\gamma_n^{\C}}_{\#}:\Map^*(\RP^m,S^{2n+1})
\stackrel{\simeq}{\rightarrow} 
\Map^*_0(\RP^m,\CP^n)$
is a homotopy equivalence.
\end{lemma}
%%%%%%%%%%%%%%%%%%
\par
{\it Remark. }
%%%%%%%%
Since $\pi_1(\Map^*(\RP^1,S^{2n+1}))=0$ and
$\pi_1(\Map^*_0(\RP^1,\CP^n))=\Z$, (ii) of Theorem \ref{lemma: Hopf}
does not hold for $m=1$.
%%%%%(Proof of Lemma 4.5)%%%
\begin{proof}
%%%
Since (i) follows from [\cite{AKY1}, Lemma 4.17],
it remains to prove (ii).
We prove it by induction on $m$.
First, assume that $m=2$, and consider the following commutative diagram
of restriction fibration sequences
%%(16)%%
\begin{equation}\label{****}
\begin{CD}
\Omega^2S^{2n+1} @>>> \Map^*(\RP^2,S^{2n+1}) @>r>> \Omega S^{2n+1}
\\
@V{\Omega^2\gamma_n^{\C}}V{\simeq}V @V{\gamma_n}_{\#}VV @V{\Omega \gamma_n^{\C}}VV
\\
\Omega^2\CP^n @>>> \Map^*_0(\RP^2,\CP^n) @>r>> \Omega \CP^n
\end{CD}
\end{equation}
%%%
Since $\Omega^2\gamma_n^{\C}$ is a homotopy equivalence
and
${\Omega \gamma_n^{\C}}_*:\pi_k(\Omega S^{2n+1})
\stackrel{\cong}{\rightarrow}\pi_k(\Omega \CP^n)$
is an isomorphism for any $k\geq 2$,
by the Five Lemma 
the induced homomorphism
${{\gamma_n^{\C}}_{\#}}_*:\pi_k(\Map^*(\RP^2,S^{2n+1}))
\stackrel{\cong}{\rightarrow} 
\pi_k(\Map^*_0(\RP^2,\CP^n))$ is an isomorphism for any $k\geq 2$.
On the other hand, because
the homotopy exact sequence of the lower row of (\ref{****})
is
$$
0\to \pi_1(\Map^*_0(\RP^2,\CP^n))\to
\Z =\pi_1(\Omega \CP^n) \stackrel{\partial}{\rightarrow}
\pi_0(\Omega^2\CP^n)=\Z \to 0,
$$
we see that
$\pi_1(\Map^*_0(\RP^2,\CP^n))=0$.
Thus, since  $\Map^*(\RP^2,S^{2n+1})$ is $(2n-2)$-connected
(by Lemma \ref{lemma: simply}),
the induced homomorphism
${\Omega \gamma_n^{\C}}_*:\pi_1(\Omega S^{2n+1})
\stackrel{\cong}{\rightarrow}\pi_1(\Omega \CP^n)$
is  an isomorphism.
Hence, $\pi_k({\gamma_n^{\C}}_{\#})$ is an isomorphism for any $k\geq 1$ and 
the assertion is true for $m=2$.
\par
Now suppose that  
${\gamma_n}_{\#}^{\p}:
\Map^*(\RP^{m-1},S^{2n+1})\stackrel{\simeq}{\rightarrow}
\Map^*_0(\RP^{m-1},\CP^n)$
is a  homotopy equivalence for some $m\geq 3$, and
consider the following commutative diagram
of restriction fibration sequences
%%(15)%%
\begin{equation*}\label{***}
%%%
\begin{CD}
\Omega^mS^{2n+1} @>>> \Map^*(\RP^m,S^{2n+1}) @>r>> \Map^*(\RP^{m-1},S^{2n+1})
\\
@V{\Omega^m\gamma_n^{\C}}V{\simeq}V @V{\gamma_n}_{\#}VV @V{\gamma_n}_{\#}^{\p}V{\simeq}V
\\
\Omega^m\CP^n @>>> \Map^*_0(\RP^m,\CP^n) @>r>> \Map^*_0(\RP^{m-1},\CP^n)
\end{CD}
\end{equation*}
%%%%%%
Since $\Omega^m\gamma_n^{\C}$ and ${\gamma^{\p}_n}_{\#}$ are homotopy equivalences,
by  the Five Lemma,
we see that
${\gamma_n}_{\#}:\Map^*(\RP^{m},S^{2n+1})\stackrel{\simeq}{\rightarrow}
\Map^*_0(\RP^{m},\CP^n)$
is also a homotopy equivalence.
%%%%%%%%%%%%%%%%%%
\end{proof}
%%%(End of proof of Lemma 4.5)%%

%%%(Corollary 4.6)%%%%%
\begin{cor}\label{cor: M0-homotopy}
%%%%
%%
If $2\leq m\leq 2n$,
$\Map^*_0(\RP^m,\CP^n)$ is $(2n-m)$-connected and
$$
\pi_{2n-m+1}(\Map^*_0(\RP^m,\CP^n))\cong
\begin{cases}
\Z & \mbox{ if }\ m\equiv 1\ \mo ,
\\
\Z/2 & \mbox{ if }\ m\equiv 0\  \mo .
\end{cases}
$$
\end{cor}
%%%%%%%%%%%%%%%%%%
\begin{proof}
This follows from Lemma \ref{lemma: simply} and
Lemma \ref{lemma: Hopf}.
\end{proof}
%%%%(End of proof of Corollary 4.6)%%

%%(Corollary 4.7)%%
\begin{cor}\label{cor: mu}
%%%%%%
If $2\leq m\leq 2n$, the map
${\mu_n}_{\#}:\Map_0^*(\RP^m,\RP^{2n+1})\stackrel{\simeq}{\rightarrow}
\Map_0^*(\RP^m,\CP^n)$
is a homotopy equivalence.
\end{cor}
%%%%%%%%%%%%%%%%%%
\begin{proof}
%%%
Consider the commutative diagram
$$
\begin{CD}
\Map^*(\RP^m,S^{2n+1}) @>{\gamma_n^{\C}}_{\#}>\simeq> \Map_0^*(\RP^m,\CP^n)
\\
@V{{\gamma_{2n+1}}_{\#}}V{\simeq}V \Vert @.
\\
\Map_0^*(\RP^m,\RP^{2n+1}) @>{\mu_n}_{\#}>> \Map_0^*(\RP^m,\CP^n)
\end{CD}
$$
Since ${\gamma_{2n+1}}_{\#}$ and ${\gamma_n^{\C}}_{\#}$ are homotopy equivalences
by Lemma \ref{lemma: Hopf}, 
the assertion easily follows from the above commutative diagram.
%%%
\end{proof}
%%End of proof of Corollary 4.5)%%%

Now we can prove Theorem \ref{thm: II} and Corollary \ref{cor: III}.

%%(Proof of Theorem 1.5)%%%%
\begin{proof}[Proof of Theorem \ref{thm: II}]
%%%
First, it is easy to see
the diagram 
%%%(16)%%%
\begin{equation}\label{basic CD}
%%%%
\begin{CD}
A_d^{\R}(m,2n+1) @>j_d^{\R}>> \Map^*(\RP^m,S^{2n+1})
\\
@V\psi_{m,n}V{\cong}V \Vert @.
\\
A_d^{\C}(m,n) @>j_d^{\C}>> \Map^*(\RP^m,\CP^n)
\end{CD}
\end{equation}
%%%%%
is commutative.
Since $j_d^{\R}$ is a homotopy equivalence through dimension
$D_{\R}(d;m,2n+1)$ if $m<2n$ and a homology equivalence through
dimension $D_{\R}(d;m,2n+1)$ if $m=2n$,
by Theorem \ref{thm: AKY1-II},so is  the map
$j_d^{\C}$.
Because $D_{\C}(d;m,n)=D_{\R}(d;m,2n+1)$,
we see that the map $j_d^{\C}$ is a homotopy equivalence through
dimension $D_{\C}(d;m,n)$ if $m<2n$ and a homology equivalence
through dimension $D_{\C}(d;m,n)$ if $m=2n$.
It remains to show that the same holds for map $i_d^{\C}$.
However, this follows easily from the facts that
$i_d^{\C}={\gamma_n^{\C}}_{\#}\circ j_d^{\C}$ and 
${\gamma_n^{\C}}_{\#}$ is a homotopy equivalence
(by Lemma \ref{lemma: Hopf}).
This completes the proof of Theorem \ref{thm: II}.
%%%%
\end{proof}
%%(End of proof of Theorem 1.5)%%%%

%%%(Proof of Corollary 1.6)%%
\begin{proof}[Proof of Corollary \ref{cor: III}]
%%%%
Since $j_{d+2}^{\C}\circ s_d=j_d^{\C}$, the assertion easily follows from
Theorem \ref{thm: II}.
%If we consider the commutative diagram
%$$
%\begin{CD}
%A^{\C}_d(m,n) @>{j_d^{\C}}>> \Map^*(\RP^m,S^{2n+1})
%\\
%@V{s_d}VV \Vert @.
%\\
%A^{\C}_d(m,n) @>{j_{d+2}^{\C}}>> \Map^*(\RP^m,S^{2n+1})
%\end{CD}
%$$
%the assertion easily follows from Theorem  \ref{thm: II}.
%%
\end{proof}
\section{The stabilized space $A_{\infty+\epsilon}^{\C}(m,n)$.  }
Although we cannot prove Conjecture \ref{conj: Psi}, we can prove the following
stabilized version.

%\par\vspace{3mm}\par
%{\bf Definition. }
%%%%%%%%%%%%
%%(Definition 5.1)%%
\begin{defi}
{\rm
For $\epsilon =0$ or $1$, let $A^{\C}_{\infty +\epsilon}(m,n)$ denote the
stabilized space
$\dis A^{\C}_{\infty +\epsilon}(m,n)=
\lim_{k\to\infty}A_{2k+\epsilon}^{\C}(m,n)$, where
the limit is taken over  the stabilization maps
$$
A_{\epsilon}^{\C}(m,n)\stackrel{s_{\epsilon}}{\longrightarrow}
A_{2+\epsilon}^{\C}(m,n)\stackrel{s_{2+\epsilon}}{\longrightarrow}
A_{4+\epsilon}^{\C}(m,n)\stackrel{s_{4+\epsilon}}{\longrightarrow}
A_{6+\epsilon}^{\C}(m,n)\stackrel{s_{6+\epsilon}}{\longrightarrow}
\cdots
$$
From the commutative diagram
$$
\begin{CD}
 @>>> A^{\C}_{2k+\epsilon}(m,n) @>s_{2k+\epsilon}>> A^{\C}_{2k+2+\epsilon}(m,n)
@>s_{2k+2+\epsilon}>> \cdots
\\
@. @V{\Psi_{2k+\epsilon}^{\C}}VV @V{\Psi^{\C}_{2k+2+\epsilon}}VV @.
\\
 @>>> \Alg^*_{2k+\epsilon}(\RP^m,\CP^n) @>\subset>>
\Alg^*_{2k+2+\epsilon}(\RP^m,\CP^n)
@>\subset>> \cdots
\end{CD}
$$
we obtain a stabilized map
$\dis
\Psi^{\C}_{\infty +\epsilon}=
\lim_{k\to\infty}\Psi_{2k+\epsilon}^{\C}:
A^{\C}_{\infty +\epsilon}(m,n)\to \Alg^*_{\epsilon}(m,n).$
%%%
}
\end{defi}

%%(Stabilized version)%%
%%%(Proposition 5.2)%%%%
\begin{prop}\label{prop: infty}
%%%
If  $2\leq m\leq 2n$ and $\epsilon =0$,
the map 
$$
\Psi_{\infty +0}^{\C}:
A^{\C}_{\infty +0}(m,n)\stackrel{\simeq}{\rightarrow} 
\Alg^*_{0}(m,n)
\simeq \Map_0^*(\RP^m,\CP^n)
$$
is a homotopy equivalence if $m<2n$ and a homology equivalence
if $m=2n$.
\end{prop}
%%%
%\begin{proof}
%%%
\par
{\it Proof. }
The assertion easily follows from
Theorem \ref{thm: stable}, Corollary  \ref{cor: IV} and the commutative diagram
$$
\begin{CD}
A^{0}_{\infty}(m,n) @>\lim_{k}i_{2k}^{\C}>\simeq> \Map_0^*(\RP^m,\CP^n)
\\
@V{\Psi_{\infty +0}^{\C}}VV \Vert @.
\\
\Alg^*_{0}(m,n) @>i>\simeq> \Map_0^*(\RP^m,\CP^n)
\qed
\end{CD}
$$
%\end{proof}
%%%(End of proof of Proposition 5.2)

%%(SECTION 6)%%%%
\section{The space $\Alg^{*}_1(\RP^m,\CP^n)$.}
%%%%%%%%%%%%%

In this section we investigate the homotopy  of 
$A_1^{\C}(m,n)\cong \Alg_1^*(\RP^m,\RP^n)$.

%%%%%%%%
%\par\vspace{2mm}\par
%{\bf Definition. }
%%%%(Definition 6.1)%%%
\begin{defi}
{\rm
For integers $1\leq m\leq 2n$, let
$V_{2n+1,m}$ denote the real Stiefel manifold  of all orthogonal $m$-frames
in $\R^{2n+1}$ and 
$({\bf b}_1,\cdots ,{\bf b}_m)\in V_{2n+1,m}$ any element such that
${\bf b}_k=\ ^t(b_{k,1},\cdots ,b_{k,2n+1})\in\R^{2n+1}$ $(1\leq k\leq m)$.
Consider
the $(n+1)$-tuple of polynomial
defined by
$$
(f_0,\cdots ,f_n)=(z_0,\cdots ,z_m)
\left(\begin{array}{cccccc}
1 & 0 & 0 &\cdots & 0 & 0
\\
\sqrt{-1}{\bf b} & {\bf c}_1 & {\bf c}_2 & \cdots & {\bf c}_{n-1} & {\bf c}_n
\end{array}
\right),
$$
%%%%%
where 
${\bf b}\in \R^m$ and ${\bf c}_k\in \C^m$ $(k=1,2,\cdots ,n)$ are given by
$$
\begin{cases}
{\bf b}=& ^{t}(b_{1,1},b_{1,2},b_{1,3},\cdots .b_{1,m}),
\\
{\bf c}_k=& ^t(b_{2k,1}+\sqrt{-1}b_{2k+1,1},b_{2k,2}+\sqrt{-1}b_{2k+1,2},\cdots ,b_{2k,m}+\sqrt{-1}b_{2k+1,m})
\end{cases}
$$
%%%%%%
Since it is easy to see that $(f_0,\cdots ,f_n)\in A_1^{\C}(m,n)$, one can define the map
$\varphi_{m,n}:V_{2n+1,m}\to A_1^{\C}(m,n)$ by
$\varphi_{m,n}({\bf b}_1,\cdots ,{\bf b}_m)=(f_0,\cdots ,f_n)$.
}
\end{defi}

%%(Lemma 6.2)%%
\begin{lemma}\label{lemma: V}
%%%%%%%%%%
If $1\leq m\leq 2n$, the map
$\varphi_{m,n}:V_{2n+1,m}\stackrel{\simeq}{\rightarrow}
A_1^{\C}(m,n)$ is a homotopy equivalence.
\end{lemma}
%%%%%
\begin{proof}
%%%
Let us consider the element
$(f_0,\cdots ,f_n)\in\C [z_0,\cdots ,z_m]^{n+1}$ of the form
%%%()%%
\begin{equation}\label{matrix}
(f_0,\cdots ,f_n)=
(z_0,\cdots ,z_m)
\left(
\begin{array}{cccc}
1 & 0 & \cdots & 0
\\
a_{1,0} & a_{1,1} & \cdots & a_{1,n}
\\
%a_{2,0} & a_{2,1} & \cdots & a_{2,n}
%\\
\vdots & \vdots & \vdots & \vdots
\\
a_{m,0} & a_{m,1} & \cdots & a_{m,n}
\end{array}
\right),
\end{equation}
%%%%
where $a_{k,j}=b_{k,j}+\sqrt{-1}c_{k,j}$ $(b_{k,j},c_{k,j}\in\R)$
and we write
$$
\begin{cases}
{\bf a}=\ ^t(b_{1,0},b_{2,0},\cdots ,b_{m,0})\in \R^m
\\
B=\big(b_{k,j}\big)_{1\leq k\leq m,1\leq j\leq n},\
C=\big(c_{k,j}\big)_{1\leq k\leq m,0\leq j\leq n}.
\end{cases}
$$
It is easy to see that
the polynomials $f_0,\cdots ,f_n$ have no common
 {\it real} root beside zero if and only if
 the equation
$$
(z_0,z_1,\cdots ,z_m)
\left(\begin{array}{cccccccc}
1 & 0 & 0 & 0 &   \cdots & 0 & 0
\\
b_{1,0}&c_{1,0} & b_{1,1} & c_{1,1} &   \cdots & b_{1,n} & c_{1,n}
\\
%b_{2,0}&c_{2,0} & b_{2,1} & c_{2,1} &   \cdots & b_{2,n} & c_{2,n}
%\\
\vdots &\vdots & \vdots  & \vdots  &    \vdots & \vdots  & \vdots 
\\
b_{m,0}&c_{m,0} & b_{m,1} & c_{m,1} &   \cdots & b_{m,n} & c_{m,n}
\end{array}
\right)=
\left(
\begin{array}{c}
0
\\
0
\\
\vdots
%\\
%\vdots
\\
0
\end{array}
\right)
$$
has no  non-zero solution.
So we see that
$(f_0,\cdots ,f_n)\in A_1^{\C}(m,n)$ if and only if 
the $\big(m\times (2n+1)\big)$-matrix $\big(B, C)$ has rank $m$.
Thus the map
$\Phi : V_{2n+1,m}\times \R^m \to A_1^{\C}(m,n)$
given by $( B,C,{\bf a})\mapsto (f_0,\cdots ,f_n)$
is clearly a homeomorphism, where $(f_0,\cdots ,f_n)$ is defined by
(\ref{matrix}).
%%%%%
Since $\varphi_{m,n}=\Phi \vert  V_{2n+1,m}\times \{{\bf 0}_m\},$
the map
$\varphi_{m,n}$ is a homotopy equivalence.
\end{proof}
%%%%(End of proof of Lemma 6.2)%%%%

%%(Lemma 6.3)%%
\begin{lemma}\label{lemma: 1-connected}
%%%
$\I$
If $2\leq m \leq 2n$,  $\Map_1^*(\RP^m,\CP^n)$ is $(2n-m)$-connected.
\par
$\II$
If $1\leq m\leq 2n$, $V_{2n+1,m}$ is $(2n-m)$-connected.
\par
$\III$
If $1\leq m \leq 2n$,
$\pi_{2n-m+1}(V_{2n+1,m})=
\begin{cases}
\Z & \mbox{if }\ m\equiv 1\  \mo ,
\\
\Z /2 & \mbox{if }\ m\equiv 0\  \mo .
\end{cases}$
\end{lemma}
%%%%%%
\begin{proof}
(i)
Consider the restriction fibration sequence
%%(18)%%
\begin{equation}\label{res-fib}
\Omega^m\CP^n \stackrel{\hat{j}}{\longrightarrow}
 \Map_1^*(\RP^m,\CP^n)\stackrel{r}{\longrightarrow}\Map_1^*(\RP^{m-1},\CP^n).
\end{equation}
%%%
We prove the assertion (i) by induction on $m$.
First, consider the case $m=2$.
Since $\pi_1(\Omega^2\CP^n)=0$ and $\Map_1^*(\RP^1,\CP^n)=\Omega\CP^n$,
taking $m=2$  in (\ref{res-fib}) and using the induced exact sequence
$$
0\to
\pi_1(\Map_1^*(\RP^2,\CP^n))\stackrel{r_*}{\longrightarrow}
\pi_1(\Omega\CP^n)=\Z\stackrel{\partial}{\rightarrow}
\Z=\pi_0(\Omega^2\CP^n)\to 0,
$$
we see that $\pi_1(\Map^*_1(\RP^2,\CP^n))=0$.
Since $\pi_k(\Omega^2\CP^n)\cong\pi_{k+2}(S^{2n+1})=0$ 
and $\pi_k(\Omega\CP^n)\cong\pi_{k+1}(S^{2n+1})=0$ for $2\leq k\leq 2n-2$,
applying (\ref{res-fib}) with $m=2$
we see that $\pi_k(\Map_1^*(\RP^2,\CP^n))=0$
for $2\leq k\leq 2n-2$.
Hence, the case $m=2$ has been proved.
%%%
Next,  assume that $\Map^*_1(\RP^{m-1},\CP^n)$ is
$(2n-m+1)$-connected for some $m\geq 3$.
Since $3\leq m <2n$, $\Omega^m\CP^n=\Omega^mS^{2n+1}$ is
$(2n-m)$-connected. 
Hence, from (\ref{res-fib})
we easily deduce that $\Map^*_1(\RP^m,\CP^n)$ is $(2n-m)$-connected.
We have proved (i).
\par
%%(ii)%%
(ii)
The assertion (ii) can be easily
proved by induction on $m$ by making use of the following
fibration sequence:
%%(19)%%
\begin{equation}\label{fibV}
S^{2n-m+1} \to V_{2n+1,m} \to V_{2n+1,m-1}.
\end{equation} 
We omit the details. 
\par
%%(iii)%%%
(iii)
First, let $m\equiv 0$ $\mo$.
It is known that $H^{2n-m+2}(V_{2n+1,m},\Z/p)=0$ for any
odd prime $p\geq 3$,
$$
H^*(V_{2n+1,m},\Z/2)=E[x_j:2n-m+1\leq j\leq 2n]
\quad (\vert x_j\vert =j)
$$
and that $Sq^1(x_{2n-m+1})=x_{2n-m+2}$.
Hence, the $(2n-m+2)$-skeleton of
$V_{2n+1,m}$ is $S^{2n-m+1}\cup_2e^{2n-m+2}$
(up to homotopy equivalence), and we have
$\pi_{2n-m+1}(V_{2n-m+1})= \Z/2.$
If $m\equiv 1$ $\mo$, (\ref{fibV})
induces the  exact sequence
$$
\pi_{N+1}(V_{2n+1,m-1})=\Z/2\stackrel{\partial}{\rightarrow}
\pi_{N}(S^{2n-m+1})=\Z\to
\pi_{N}(V_{2n+1,m}) \to 0,
$$
where $N=2n-m+1$.
Hence, $\pi_{2n-m+1}(V_{2n+1,m})\cong\Z$.
Thus (iii) has also been proved.
%%%%%%
\end{proof}
%%(End of Proof of Lemma 6.3)%%

%%%(Corollary 6.4)%%%%
\begin{cor}\label{cor: pi1}
%%%
If $1\leq m \leq 2n$, 
$\Alg_1^*(\RP^m,\CP^n)$ is $(2n-m)$-connected and
$$
\pi_{2n-m+1}(\Alg_1^*(\RP^m,\CP^n))=
\begin{cases}
\Z & \mbox{if }\ m\equiv 1\  \mo ,
\\
\Z /2 & \mbox{if }\ m\equiv 0\  \mo .
\end{cases}
$$
\end{cor}
\begin{proof}
Since there is a homotopy equivalence
$V_{2n+1,m}\simeq \Alg_1^*(\RP^m,\CP^n)$, the assertion
follows from Lemma \ref{lemma: 1-connected}.
\end{proof}
%%(End of Corollary 6.4)%%%

%%
%%(Definition 6.5)%%%%%%
%\par\vspace{2mm}
%\par
%{\bf Definition. }
\begin{defi}
{\rm
For $1\leq m\leq 2n$,
define the map $\tilde{i}_m:V_{2n+1,m}\to \Map_1^*(\RP^m,\CP^n)$
by
%%( )%%
%\begin{equation}\label{tilde-i}
%%%%
$\tilde{i}_m=i_1^{\C}\circ \varphi_{m,n}.$
%\end{equation}
%%%%%%
}
\end{defi}

For $2\leq m\leq 2n$, it is easy to verify that the following diagram is commutative
%%(20)%%
\begin{equation}\label{fibCD}
%%%%
\begin{CD}
S^{2n-m+1} @>\tilde{j}>> V_{2n+1,m} @>>> V_{2n+1,m-1}
\\
@V{\hat{s}_m}VV @V{\tilde{i}_m}VV @V{\tilde{i}_{m-1}}VV
\\
\Omega^m \CP^n @>{\hat{j}}>>
\Map_1^*(\RP^m,\CP^n) @>r>>
\Map_1^*(\RP^{m-1},\CP^n)
\end{CD}
\end{equation}
%%%
where we identify
$S^{2n-m+1}\cong \frac{O(2n+2-m)}{O(2n+1-m)}$ and the two rows are
fibration sequences.

%%(Lemma 6.6)%%%
\begin{lemma}\label{lemma: s}
%%%%%%%%
If $2\leq m\leq 2n$,
the map $\hat{s}_m:S^{2n-m+1}\to \Omega^m\CP^n$ is a homotopy equivalence
up to dimension $D_{\C}(1;m,n)=4n-2m+1$.
\end{lemma}
%%%
\begin{proof}
By means of a method similar to the one used in the proof of
[\cite{Y12}, Lemma 3.1]
we can show that 
$\hat{s}_{m*}:\pi_{2n-m+1}(S^{2n-m+1})
\stackrel{\cong}{\rightarrow}\pi_{2n-m+1}(\Omega^mS^{2n+1})$
$\cong
\pi_{2n-2m+1}(\Omega^m\CP^n)$
is an isomorphism.
Thus we can identify  $\hat{s}_m$ with the $m$-fold suspension 
 $E^m :S^{2n-m+1}\to\Omega^mS^{2n+1}\simeq \Omega^m\CP^n$
(up to homotopy equivalence).
Hence $\hat{s}_m$ is a homotopy equivalence up to dimension
$4n-2m+1$.
\end{proof}
%%(End of Lemma 6.6)%%%

%%(Lemma 6.7)%%
\begin{lemma}\label{lemma: m1}
%%%
If $n\geq 2$, %the induced homomorphism
$\tilde{i}_{1*}:\pi_k(V_{2n+1,1})\to \pi_k(\Map_1^*(\RP^1,\CP^n))$
is an isomorphism for any $2\leq k <4n-1$ and an epimorphism
for $k=4n-1$.
\end{lemma}
%%%%%%
\begin{proof}
%%%%
After identifications (up to homotopy equivalence) $V_{2n+1,1}=S^{2n}$ and
$\Map_1^*(\RP^1,\CP^n)=\Omega \CP^n= \Omega S^{2n+1}\times S^1$,
the map $\tilde{i}_1$ can be viewed as a map
$\tilde{i}_1:S^{2n}\to\Omega\CP^n$.
%%%%%
Let $q_1 :\Omega\CP^n=\Omega S^{2n+1}\times S^1\to \Omega S^{2n+1}$ denote
the projection onto the first factor.
%%%
Note that the composite
$q_1\circ \tilde{i}_1$ can be identified with the map
$S^{2n}\simeq Q^1_{(2n+1)}(\R)\to \Omega S^{2n+1}$
given in [\cite{KY1}, Corollary 5 (2)].
Hence it follows from [\cite{KY1}, Corollary 5] that $q_1\circ \tilde{i}_1$
 is a homotopy equivalence up to dimension
$N(1,2n+2)=4n-1$.
Recalling  that
 ${q_1}_*:\pi_l(\Omega \CP^n)\stackrel{\cong}{\rightarrow}\pi_l(\Omega S^{2n+1})$
is an isomorphism for all $l\geq 2$,
we see that $\pi_k(\tilde{i}_1)$ is an isomorphism for any $2\leq k <4n-1$ and
an epimorphism for $k=4n-1$.
\end{proof}
%%%(End of Proof of Lemma 6.7)%%

%%(Proposition 6.8)%%
\begin{prop}\label{prop: V}
%%%%%%%%%%%%%%%%%%%%%
If $2\leq m\leq 2n$, the map
$\tilde{i}_m:V_{2n+1,m}\to \Map_1^*(\RP^m,\CP^n)$
is a homotopy equivalence up to dimension $D_{\C}(1;m,n)=4n-2m+1$
\end{prop}
\begin{proof}
%%%
The proof proceeds by  induction on $m$.
First, consider the case $m=2$.
Since
$\pi_1(V_{2n+1,2})=\pi_1(\Map_1^*(\RP^2,\CP^n))=0$
by Lemma \ref{lemma: 1-connected},  the map
$\tilde{i}_2$ induces an isomorphism on $\pi_1(\ )$.
Hence it suffices to show that
$\tilde{i}_{2*}:\pi_k(V_{2n+1,2})\to \pi_k(\Map_1^*(\RP^2,\CP^n))$
is an isomorphism for any $2\leq k<D_{\C}(1;2,n)=4n-1$ and
an epimorphism for $k=D_{\C}(1;2,n)$.
However, recalling the commutative diagram (\ref{fibCD}) for $m=2$,
and using the Five Lemma,
 Lemma \ref{lemma: s} and Lemma \ref{lemma: m1}, we see that
 $\pi_k(\tilde{i}_2)$ is an isomorphism for any $2\leq k<4n-1=D_{\C}(1;2,n)$ and
 an epimorphism for $k=4n-1=D_{\C}(1;2,n)$.
Hence, the case $m=2$ is proved.
\par
Now assume that the map $\tilde{i}_{m-1}$ is a homotopy equivalence up to
dimension $D_{\C}(1;m-1,n)=4n-2m+3$ for some $m\geq 3$.
Since, by Lemma \ref{lemma: s}, $\hat{s}_m$ is a homotopy equivalence up to dimension $D_{\C}(1;m,n)=4n-2m+1$, the Five Lemma and the diagram (\ref{fibCD}) imply that
the map $\tilde{i}_m$ is a homotopy equivalence up to dimension $D_{\C}(1;m,n)$.
%%%
\end{proof}
%%(End of proof of Proposition 6.8)%%%
%
%%(Lemma 5.6)%%
%\begin{lemma}\label{lemma: V2}
%%%%%%%%%%
%If $1\leq m \leq 2n$,
%$\pi_{2n-m+1}(V_{2n+1,m})=
%\begin{cases}
%\Z & \mbox{if }\ m\equiv 1\  \mo
%\\
%\Z /2 & \mbox{if }\ m\equiv 0\  \mo
%\end{cases}$
%\end{lemma}
%%%%%%%%%%%
%\begin{proof}
%%%%%%
%First, consider the case $m\equiv 0$ $\mo$.
%In this case, it is known that $H^{2n-m+2}(V_{2n+1,m},\Z/p)=0$ for any
%odd prime $p\geq 3$,
%$$
%H^*(V_{2n+1,m},\Z/2)=E[x_j:2n-m+1\leq j\leq 2n]
%\quad (\vert x_j\vert =j)
%$$
%and that $Sq^1(x_{2n-m+1})=x_{2n-m+2}$.
%Hence, the $(2n-m+2)$-skeleton of
%$V_{2n+1,m}$ is $S^{2n-m+1}\cup_2e^{2n-m+2}$
%(up to homotopy equivalence), and we have
%$\pi_{2n-m+1}(V_{2n-m+1})= \Z/2.$
%If $m\equiv 1$ $\mo$ and  set $N=2n-m+1$, by using (\ref{fibV})
%we have the  exact sequence
%$$
%\pi_{N+1}(V_{2n+1,m-1})=\Z/2\stackrel{\partial}{\rightarrow}
%\pi_{N}(S^{2n-m+1})=\Z\to
%\pi_{N}(V_{2n+1,m}) \to 0.
%$$
%Hence, $\pi_{2n-m+1}(V_{2n+1,m})\cong\Z$.
%This completes the proof.
%%%%%
%\end{proof}
%%(End of proof of Lemma 5.6)%%%

%%(Proof of Theorem 1.8)%%
\par
{\it Proof of Theorem \ref{thm: V}. }
%%%%%%%
(i)
Since $i_1^{\C}=i_{1,\C}\circ \Psi_1^{\C}$ and the projection map
$\Psi_1^{\C}:A_1^{\C}(m,n)
\stackrel{\cong}{\rightarrow} \Alg_1^*(\RP^m,\CP^n)$
is a homeomorphism, it suffices to show that
$i_1^{\C}$ is a homotopy equivalence up to dimension $D_{\C}(1;m,n).$
As $\tilde{i}_m=i^{\C}_1\circ \varphi_{m,n}$ and
$\varphi_{m,n}$ is a homotopy equivalence,
this follows from Proposition \ref{prop: V}.
Hence, (i) has been proved.
\par
%%(ii)%%
(ii)
Since
$\hat{s}_{2n*}:\pi_1(S^1)\to \pi_1(\Omega^{2n}\CP^n)$ is an
epimorphism by Lemma \ref{lemma: s}
and $\pi_1(S^1)\cong\Z\cong\pi_1(\Omega\CP^n)$,
$\hat{s}_{2n}$ induces an isomorphism on $\pi_1(\ )$.
Consider the following commutative diagram 
of the exact sequences induced from 
(\ref{fibCD}) for $m=2n$
$$
\begin{CD}
\pi_2(V_{2n+1,2n-1}) @>\partial>>\pi_1(S^1) @>>> \pi_1(V_{2n+1,2n})=\Z/2@>>> 0
\\
@V{\tilde{i}_{2n-1*}}V{\cong}V @V{\hat{s}_{2n*}}V{\cong}V 
@V{\tilde{i}_{2n*}}VV @.
\\
\pi_2(\Map_1^*) @>\partial^{\p}>> \pi_1(\Omega\CP^n) @>>> \pi_1(\Map_1^*(\RP^{2n},\CP^n)) @>>> 0
\end{CD}
$$
where $\Map_1^*=\Map_1^*(\RP^{2n-1},\CP^n)$.
Since, by Proposition \ref{prop: V}, the induced homomorphism
$\tilde{i}_{2n-1*}:\pi_2(V_{2n+1,2n-1})
\stackrel{\cong}{\rightarrow}\pi_2(\Map_1^*(\RP^{2n-1},\CP^n))$
is an isomorphism, $\tilde{i}_{2n}$ induces an isomorphism on $\pi_1(\ )$.
\qed
%%End of Proof of Theorem 1.8)%%

%%(Proof of Corollary 1.9)%%
\begin{proof}[Proof of Corollary \ref{cor: V}]
%%%%
The assertion (i)  follows from Lemma \ref{lemma: simply}
and Theorem \ref{thm: II}.
The assertion (ii) also easily follows from Corollary \ref{cor: M0-homotopy},
Lemma \ref{lemma: V}, Lemma \ref{lemma: 1-connected},
Proposition \ref{prop: V}
and Theorem \ref{thm: V}.
%%
%%%%%%
\end{proof}
%%(End of proof of Corollary 1.9)%%

\bibliographystyle{amsplain}

%%%(References)%%%%%%%

%%%%%%%%%%%%%%%%%%%%%%%%%%

\end{document}